\documentclass[onefignum,onetabnum]{siamart220329}

\usepackage{lipsum}
\usepackage{amsfonts}
\usepackage{tikz-cd}  
\usepackage{amsmath,booktabs,ctable,threeparttable}
\usepackage{amssymb,amsfonts,boxedminipage}
\usepackage{graphicx,epstopdf} 
\usepackage{cite}
\usepackage{algorithm,algpseudocode}
\usepackage{arydshln}
\usepackage{lineno}
\usepackage[caption=false]{subfig} 
\graphicspath{{fig/}{fig/dynamics/}}

\ifpdf
  \DeclareGraphicsExtensions{.eps,.pdf,.png,.jpg}
\else
  \DeclareGraphicsExtensions{.eps}
\fi

\algrenewcommand\algorithmicrequire{\textbf{Input:}}
\algrenewcommand\algorithmicensure{\textbf{Output:}}

\newsiamremark{remark}{Remark}
\newsiamremark{hypothesis}{Hypothesis}
\crefname{hypothesis}{Hypothesis}{Hypotheses}
\newsiamthm{claim}{Claim}

\headers{Derivative estimation by RKHS}{H. Guo and H. Li}

\title{Derivative estimation by RKHS regularization for learning dynamics from time-series data }

\author{Hailong Guo\thanks{School of Mathematics and Statistics, The University of Melbourne, Parkville, VIC 3010, Australia. 
  (\email{hailong.guo@unimelb.edu.au}).}
\and Haibo Li \thanks{School of Mathematics and Statistics, The University of Melbourne, Parkville, VIC 3010, Australia.  
  (\email{haibo.li@unimelb.edu.au}).}
  }

\usepackage{amsopn}
\DeclareMathOperator{\diag}{diag}
\usepackage{multirow}

\newcommand\bx{\boldsymbol{x}}

\newcommand\bbf{\boldsymbol{f}}
\newcommand\by{\boldsymbol{y}}

\newcommand\bv{\boldsymbol{v}}
\newcommand\bu{\boldsymbol{u}}
\newcommand\bw{\boldsymbol{w}}
\newcommand\bb{\boldsymbol{b}}
\newcommand\bbeta{\boldsymbol{\eta}}
\newcommand\bphi{\boldsymbol{\phi}}
\newcommand\bvarphi{\boldsymbol{\varphi}}

\newcommand\calH{\mathcal{H}}
\newcommand\calL{\mathcal{L}}

\newcommand\rd{\mathrm{d}}
\newcommand\rvec{\mathrm{vec}}

\newcommand\argmax{\mathop{\text{argmax}}}

\ifpdf
\hypersetup{
  pdftitle={Derivative estimation by RKHS},
  pdfauthor={H. Guo, and H. Li}
}
\fi

\begin{document}

\maketitle

\begin{abstract}
Learning the governing equations from time-series data has gained increasing attention due to its potential to extract useful dynamics from real-world data. Despite significant progress, it becomes challenging in the presence of noise, especially when derivatives need to be calculated. To reduce the effect of noise, we propose a method that simultaneously fits both the derivative and trajectory from noisy time-series data. Our approach formulates derivative estimation as an inverse problem involving integral operators within the forward model, and estimates the derivative function by solving a regularization problem in a vector-valued reproducing kernel Hilbert space (vRKHS). We derive an integral-form representer theorem, which enables the computation of the regularized solution by solving a finite-dimensional problem and facilitates efficiently estimating the optimal regularization parameter. By embedding the dynamics within a vRKHS and utilizing the fitted derivative and trajectory, we can recover the underlying dynamics from noisy data by solving a linear regularization problem. Several numerical experiments are conducted to validate the effectiveness and efficiency of our method.
\end{abstract}

\begin{keywords}
  dynamical system, system identification, derivative estimation, vector-valued reproducing kernel Hilbert space, integral-form representer theorem, Tikhonov regularization
\end{keywords}

\begin{MSCcodes}
37M10, 65P99, 65R32
\end{MSCcodes}

\section{Introduction}
Dynamical systems are prevalent in various fields, including engineering \cite{dixon2003nonlinear,sastry2013nonlinear,kamalapurkar2018reinforcement}, biology \cite{freedman1980deterministic,janson2012non,favela2020dynamical}, economics \cite{brock1991non}, physics \cite{arnol2013mathematical,fuchs2014nonlinear}, and mathematics \cite{katok1995introduction,perko2013differential}. However, in many cases, systems driven by physical principles may have unknown parameters, such as the mass or size of mechanical components, or their dynamics may be completely unknown. In these scenarios, system identification techniques are employed to estimate the system's dynamics from the data it generates \cite{Nelles2001,keesman2011system}. Consider the following autonomous dynamical system:
\begin{equation}\label{update1}
	\begin{cases}
	\dot{\bx}(t) = \bbf(\bx(t))  \\
	\bx(0) = \bx_0 ,
	\end{cases}
\end{equation}
where $\bx:[0,T]\rightarrow\mathbb{R}^{d}$ is the system state and  $\bbf: \mathbb{R}^{d}\rightarrow \mathbb{R}^{d}$ is the underlying dynamics. Suppose we have $n$ noisy observational states $\{(t_i,\by_i)\}_{i=1}^{n}$ with $\by_i=\bx(t_i)+\bbeta_i$, $t_i\in[0,T]$, where $\bbeta_i\sim\mathcal{N}(0,\delta^2I_d)$ are i.i.d. Gaussian white noise and $I_d$ is the identity matrix of order $d$. The system identification aims to learn the dynamics $\bbf$ from those time-series observations.

As the collection of massive datasets grows in the era of big data, discovering the underlying dynamics within these data has become a key objective in scientific research. In past years, there have been many attempts to extract dynamical equations from time-series data. For example, Bongard and Schmidt applied symbolic regression to identify nonlinear differential equations \cite{bongard2007automated,schmidt2009distilling}. A common approach represents $\bbf$ as a parameterized vector field and recovers the parameters by integrating the solution along the vector field starting from guessed initial conditions and then comparing the resulting trajectories with the observed time-series data. For example, ODENet represents $\bbf$ as a linear combination of polynomials of various orders and determines the coefficients through sparse regression \cite{hu2022revealing}; alternatively, $\bbf$ can be modeled using the units of a deep neural network, as demonstrated in methods proposed in \cite{chen2018neural,qin2019data}. In \cite{lahouel2024learning}, it is assumed that $\bbf$ belongs to a vector-valued reproducing kernel Hilbert space (vRKHS) and can be recovered by solving a constrained optimization problem, which is derived from matching the Euler approximations of the ODE solution with the observed data.

On the other hand, if the values of the derivative function $\dot{\bx}(t)$ at $\{t_i\}_{i=1}^{n}$ are known, fitting the dynamics from the noisy observation reduces to solving a regression problem, where usually a regularized problem with loss term $\sum_{i=1}^{n}\|\by_i-\dot{\bx}(t_i)\|_{2}^{2}$ should be solved. There are many such methods for recovering $\bbf$ that assume that $\{\dot{\bx}(t_i)\}_{i=1}^{n}$ are available or can be easily estimated; see e.g. \cite{wang2011predicting,zhang2018robust,wu2019numerical,wu2020structure,kang2021ident}. Among these methods, the Sparse Identification of Nonlinear Dynamics (SINDy) method is a representative and successful framework, which assumes that $\bbf$ is a linear combination of a set of basic functions \cite{brunton2016discovering}. It first estimates $\{\dot{\bx}(t_i)\}_{i=1}^{n}$ by some noisy removal algorithms, and then recover $\bbf$ by solving a linear sparse regression problems. In many real-world applications, however, the observation data are often unevenly distributed and heavily corrupted by noise. This presents a significant challenge in reliably estimating the derivative values, which in turn severely limits the algorithm's robustness to noise and, consequently, its applicability to real-world data.

Various methods have been proposed to address this challenge. In SINDy, in addition to the finite difference method, a total variation (TV) regularization algorithm is used to estimate the derivatives at the given time points. In \cite{messenger2021weak}, a method called WEAK SINDy was proposed to deal with observational noise, where $\{\dot{\bx}(t_i)\}_{i=1}^{n}$ are not required. Instead, a weak-form ODE is derived by using carefully designed test functions that help smooth out the random noise through the calculation of integrals. In \cite{hokanson2023simultaneous}, a method that simultaneously identifies and denoises a dynamical system was proposed, where a discretized ODE-constrained optimization problem needs to be solved. In comparison, the method in \cite{rudy2019deep} uses a neural network to model $\bbf$, learning the dynamics and estimating the observational noise simultaneously. Based on the Runge-Kutta method for solving ODEs, the method proposed in \cite{rudy2019smoothing} can simultaneously denoise the data and recover the parameters of the dynamics. Most of these methods require solving a nonlinear optimization problem, where at each step, the ODE needs to be solved explicitly or implicitly. While these methods are generally effective, they often involve solving complex optimization problems or repeatedly solving the ODE, which significantly increases the computational overhead.

Our goal is to estimate derivatives from noisy data as the first step, since more accurate derivative estimation improves the ability to develop dynamic or statistical models and generate precise forecasts. One major challenge in estimating derivatives is that it is an ill-posed problem, meaning that the noise in the data can be significantly amplified when using simple methods like finite differences to calculate the derivatives. In addition to the aforementioned method based on TV regularization \cite{chartrand2011numerical}, there are various approaches for estimating derivatives of noisy data. One such method involves applying a smoothing filter to the data first, followed by a finite difference calculation \cite{butterworth1930theory}. The second method relies on building a local model of the data using linear regression and then analytically calculating the derivative based on this model \cite{savitzky1964smoothing, belytschko1996meshless}. The third one is based on the Kalman filter with unknown noise and process characteristics \cite{crassidis2004optimal,aravkin2017generalized}. By formulating the derivative calculation as an inverse problem with an integration operator as the forward model, several methods have been  proposed to estimate the derivatives by solving a Tikhonov regularization problem with some appropriate smoothing penalties; see e.g. \cite{cullum1971numerical,lu2006numerical,chartrand2011numerical}. It is important to note that selecting an appropriate regularization parameter is crucial for obtaining a satisfactory regularized solution, and unfortunately, determining the optimal regularization parameter is not a trivial task. For example, the method in \cite{chartrand2011numerical} relies on grid search to determine a good regularization parameter, which requires solving the regularization problem multiple times with different parameter values in order to select the best one, which can be computationally expensive. Additionally, most of these methods are limited to handling 1-dimensional data, which means for the data $\{(t_i,\by_{i})\}_{i=1}^{n}$ with $\by_i\in\mathbb{R}^{d}$, there are $d$ regularization problems that need to be solved, one for each dimension of the data. 

In this paper, we estimate the values of $\dot{\bx}(t)$ at $\{t_{i}\}_{i=1}^{n}$ for a continuous $\dot{\bx}(t)\in\mathbb{R}^{d}$ by solving the regularized ill-posed inverse problem:
\begin{equation}\label{regu0}
  \min_{\bphi\in\calH}\sum_{i=1}^{n}\left\|\bx_0+\int_{0}^{t_i}\bphi(s)\rd s-\by_i\right\|_{2}^2+\lambda\|\bphi\|_{\calH}^2 ,
\end{equation}
where $\bphi(s):=\dot{\bx}(t)$ and $\calH$ is the hypothesis space. This formulation follows the approach in \cite{chartrand2011numerical}, but $\dot{\bx}(t)$ is allowed to be vector-valued, and $\calH=\calH_K$ is chosen as an $\mathbb{R}^{d}$-valued reproducing kernel Hilbert space (RKHS) with a reproducing kernel $K$ \cite{carmeli2006vector,carmeli2010vector}. To obtain the solution of \cref{regu0}, we first establish an integral-form representer theorem for the vRKHS, which shows that the solution must be a linear combination of a finite number of basis functions, where the basis functions depend on both the reproducing kernel and the data. With the help of the integral-form representer theorem, we only need to solve a finite-dimensional linear regularization problem to obtain the regularized solution. We demonstrate that by using commonly employed kernels, such as the Gaussian or Matérn kernels, the integral computation required for forming the Gram matrix can be avoided. Furthermore, the optimal regularization parameter can be efficiently estimated using the L-curve method \cite{hansen1999curve}. The resulting algorithm, which requires solving \cref{regu0} only once and does not require computing any integrals, can simultaneously estimate both the values of $\dot{\bx}(t)$ and $\bx(t)$ at $\{t_i\}_{i=1}^{n}$.

With the estimated $\{\dot{\bx}(t_i)\}_{i=1}^{n}$ and $\{\bx(t_i)\}_{i=1}^{n}$, we can reconstruct $\bbf$ following the procedure of SINDy. Additionally, we can employ a nonparametric inference method to model $\bbf$ as an element in an $\mathbb{R}^{d}$-valued RKHS. This approach eliminates the need to specify the function form of $\bbf$ using a dictionary of pre-defined basis functions, offering greater flexibility in modeling the dynamics. The vector field $\bbf$ can be recovered by solving a linear vRKHS regularization problem, where the representer theorem for vRKHS \cite{kadri2016operator} enables us to compute the solution by solving a finite-dimensional regularization problem, where the optimal regularization parameter can be efficiently determined by the L-curve method. This approach consists of two steps: first denoising the data, and then learning the dynamics. Each step requires solving only a finite-dimensional linear Tikhonov regularization problem. In contrast to other methods, this approach avoids many complex computations, making the algorithm more efficient and straightforward to implement.

The rest of this paper is organized as follows: In \Cref{sec2}, we provide a brief review of SINDy and the derivative estimation method employed, as well as the basic properties of vRKHS. In \Cref{sec3}, we present our method for estimating derivatives. In \Cref{sec4} we discuss the practical computations for estimating $\{\dot{\bx}(t_i)\}_{i=1}^{n}$ and $\{\bx(t_i)\}_{i=1}^{n}$ simultaneously. In \Cref{sec5}, we describe how to use these estimates to reconstruct the dynamics by embedding $\bbf$ in a vRKHS. Numerical experiments are presented in \Cref{sec6}, and the conclusion is provided in \Cref{sec7}.

\paragraph{Notations}
The notations frequently referenced throughout the paper are listed in \Cref{tab0}.
\begin{table}[tbh] 
  \caption{Table of notations. }  \label{tab0}\vspace{-2mm}
  \centering
\begin{tabular}{ c ll }
\toprule
Notation & Description \\
\toprule
$\bbf, \bphi, \bvarphi $ & $\mathbb{R}^{d}$-valued maps \\
\hline $x_i, \bx, \mathbf{X}$ & $x_i$ is a scalar/function, $\bx=(x_{1},\dots,x_{n})^{\top}$, $\mathbf{X}=(\bx_1,\dots,\bx_n)^{\top}$ \\
\hline $y_i, \by,\mathbf{Y}$ & $y_i$ is a scalar/function, $\by=(y_{1},\dots,y_{n})$, $\mathbf{Y}=(\by_1,\dots,\by_n)^{\top}$ \\ 
\hline $b_i,\bb,\mathbf{B}$ & $b_i$ is a vector, $\bb=(b_{1}^{\top},\dots,b_{n}^{\top})^{\top}$, $\mathbf{B}=(b_1,\dots,b_n)$ \\
\hline $v_i,\bv,\mathbf{V}$ & $v_i$ is a coefficient vector, $\bv=(v_{1}^{\top},\dots,v_{n}^{\top})^{\top}$, $\mathbf{V}=(v_1,\dots,v_n)$ \\ 
\hline $G_{ij},\mathbf{G}$ & $G_{ij}$ is a $d$-by-$d$ matrix, $\mathbf{G}=(G_{ij})_{1\leq i,j\leq n}$ \\    
\hline $\mathbf{0}, I_n$ & zero vector/matrix, identity matrix of order $n$ \\            
\bottomrule
\end{tabular}
\end{table}

\section{Preliminary}\label{sec2}
First, we provide a brief review of the SINDy framework for learning dynamics from data, along with the TV regularization method for estimating derivatives from noisy data. Then, we review some basis properties of the vRKHS.

\subsection{SINDy and TV derivative estimation method}
The SINDy algorithm \cite{brunton2016discovering} has been shown to be successful in learning a sparsely
represented nonlinear dynamics when noise is small and dynamic scales do not vary across multiple orders of magnitude. Suppose $\bbf=(f_1,\dots,f_d):\mathbb{R}^{d}\to\mathbb{R}^{d}$ with $f_i:\mathbb{R}^{d}\rightarrow\mathbb{R}$. This framework first chooses a dictionary of basis functions $\{\psi_j\}_{j=1}^{J}$ with $\psi_j:\mathbb{R}^{d}\rightarrow\mathbb{R}$, and assumes that $\bbf$ can be represented componentwisely by 
\begin{equation}\label{candbasis}
f_{i}(\bx)=\sum_{j=1}^{J}w_{ji}\psi_{j}(\bx), \ \ i=1,\dots,d,
\end{equation}
where $\mathbf{W}=(w_{ji})\in\mathbb{R}^{J\times d}$ composed by the coefficients should be a sparse matrix. Substituting \cref{candbasis} into \cref{update1} with data $\{(t_i,\bx(t_i))\}_{i=1}^{n}$ leads to
\[
\dot{\mathbf{X}} = \Theta(\mathbf{X})\mathbf{W},
\]
where $\mathbf{X}=(\bx(t_1),\dots,\bx(t_n))^{\top}\in\mathbb{R}^{n\times d}$, $\dot{\mathbf{X}}=(\dot{\bx}(t_1),\dots,\dot{\bx}(t_n))^{\top}\in\mathbb{R}^{n\times d}$ and $\Theta(\mathbf{X})=(\psi_{j}(\bx(t_i)))\in\mathbb{R}^{n\times J}$. Given the noisy observations $\{(t_i, \by_i)\}_{i=1}^{n}$, SINDy uses the noisy data $\mathbf{Y}=(\by_1,\dots,\by_{n})^{\top}$ to form the matrix $\Theta(\mathbf{Y})$ and uses the time derivative $\dot{\mathbf{Y}}$ estimated from $\mathbf{Y}$ to replace $\dot{\mathbf{X}}$. Then it obtains the sparse coefficients $\{w_{ji}\}$ by solving the least squares problem
\begin{equation*}
  \min_{\mathbf{W}\in\mathbb{R}^{J\times d}}\|\dot{\mathbf{Y}}-\Theta(\mathbf{Y})\mathbf{W}\|_{F}^2
\end{equation*}
by the sequentially-thresholded least squares method. Here $\|\cdot\|_{F}$ is the Frobenius norm of a matrix.

Due to the presence of observational noise, simple methods like finite differences are not robust for estimating time derivatives from $\{(t_i,\by_i)\}_{i=1}^{n}$. SINDy leverages the method from \cite{chartrand2011numerical} to estimate the derivatives, which utilizes TV regularization to improve robustness in the presence of noise. Given a function $g:[0,T]\rightarrow\mathbb{R}$ (for convenience assume $g(0)=0$) where $g\in L^{2}([0,T])$, this method computes the derivative of $g$ by solving the TV regularization problem
\begin{equation}\label{TV_deriv}
  \min_{u\in \mathrm{BV}([0,T])} \frac{1}{2}\int_{0}^{T}|Au-g|^2\rd t + \lambda \mathrm{TV}(u),
\end{equation}
where $(Au)(t):=\int_{0}^{t}u(s)\rd s$, $\mathrm{TV}(u)=\int_{0}^{T}|u'(t)|\rd t$, $\mathrm{BV}([0,T])$ is the space of functions of bounded variation, and $\lambda$ is the regularization parameter. Using the TV regualrization method, the noise can be suppressed in the derivative, and it does not suppress jump discontinuities, which allows for the computation of discontinuous derivatives. We remark that the TV method is more proper to estimate a piecewise constant derivative function, as it enforces the reconstruction of a function with bounded variation.

For the time-series noisy data $\{(t_i,\by_i)\}_{i=1}^{n}$, the integral in \cref{TV_deriv} should be discretized followed by an iterative method \cite{vogel1996iterative,vogel2002computational}. In practical computations, determining an appropriate value for $\alpha$ is challenging, and the regularization problem needs to be solved multiple times with different values of $\lambda$ to select the best one. Additionally, for $d$-dimensional data, $d$ separate regularization problems must be solved, one for each dimension of the data. This process can be computationally expensive and time-consuming.

\subsection{Vector-valued RKHS}
We review several fundamental properties of the vRKHS, which will be utilized in the subsequent sections. For simplicity, we only review the vRKHS on $\mathbb{R}$, while the definition and properties on $\mathbb{C}$ are similar. For a detailed theoretical treatment, we refer the reader to \cite{carmeli2006vector, carmeli2010vector, kadri2016operator}. For any nonempty sets $X$ and $Y$, we use $Y^{X}$ to denote the set of all maps from $X$ to $Y$.

\begin{definition}[Reproducing kernel]
  Let $X$ be a compact metric space. The map $K:X\times X \rightarrow \mathbb{R}^{d\times d}$ is called an $\mathbb{R}^d$-valued reproducing kernel on $X$ if 
  \begin{enumerate}
    \item[(1).] for any $x, x'\in X$ it holds $K(x,x')=K(x',x)^{\top}$;
    \item[(2).] for any $x_{1},\dots,x_{N}\in X$ and $z_{1},\dots,z_{N}\in\mathbb{R}^{d}$ it holds
    \[\sum_{i=1}^{N}\langle K(x_i, x_j)z_i, z_j \rangle_2 \geq 0 , \]
    where $\langle\cdot,\cdot\rangle_2$ is the Euclidean inner product on $\mathbb{R}^{d}$.
  \end{enumerate}
\end{definition}

\begin{definition}[Vector-valued RKHS]\label{def:vrkhs}
   A Hilbert space $\calH_K\subseteq (\mathbb{R}^d)^{X}$ is called an $\mathbb{R}^d$-valued reproducing kernel Hilbert space if there exist an $\mathbb{R}^d$-reproducing kernel $K:X\times X \rightarrow \mathbb{R}^{d\times d}$ such that
   \begin{enumerate}
    \item[(1).] for any fixed $v\in\mathbb{R}^d$ and $x'\in X$, the map $\varphi: x \mapsto  K(x,x')v$ belongs to $\calH_K$;
    \item[(2).] for every $x\in X$, $v\in\mathbb{R}^d$ and $\varphi\in\calH_K$, it holds $\langle \varphi, K(x,\cdot)v\rangle_{\calH_K}=\langle\varphi(x),v\rangle_2$, called the reproducing property.
   \end{enumerate}
\end{definition}

The following well-known theorem is a generalization for the real-valued RKHS, establishing a one-to-one correspondence between the vector-valued reproducing kernel and the vector-valued RKHS.

\begin{theorem}[Moore-Aronszajn]
  Suppose $K$ is an $\mathbb{R}^d$-valued reproducing kernel on $X$. Then there is a unique $\mathbb{R}^d$-valued RKHS $\calH\subseteq (\mathbb{R}^d)^{X}$ for which $K$ is the reproducing kernel.
\end{theorem}

The following representer theorem plays a key role in learning vector-valued functions from finite data using the kernel method. For a more general discussion of the representer theorem for vRKHS, see, e.g. \cite{wahba1990spline, wendland2004scattered}.

\begin{theorem}[Representer theorem]\label{thm:rt}
  Let $\calH_K\subseteq (\mathbb{R}^d)^{X}$ be an $\mathbb{R}^d$-valued RKHS with kernel $K$. Suppose we have data pairs $(x_1,z_1),\dots,(x_n,z_n)\in X\times\mathbb{R}^d$, and $\lambda>0$. Then the regularized optimization problem 
  \begin{equation}\label{regu2}
    \min_{\bvarphi\in\calH_K} J(\bvarphi):= \sum_{i=1}^{n}\|\bvarphi(x_i)-z_i\|_{2}^2 + \lambda\|\bvarphi\|_{\calH_K}^2
  \end{equation}
  has at least one solution, and any solution must have the representation
  \begin{equation}
    \bvarphi^{*} = \sum_{i=1}^{n}K(x_i,\cdot)v_i
  \end{equation}
  with some $v_i\in\mathbb{R}^d$ for $i=1,\dots,n$.
\end{theorem}

The representer theorem is particular useful for solving the regularization problem \cref{regu2}, since it transforms this infinite-dimensional optimization problem to a finite-dimensional problem. We only need to compute the coefficient vector $\{v_{i}\}_{i=1}^{n}$ to get the regularized solution.

\section{Estimating derivative function by vRKHS regularization}\label{sec3}
For the ODE \cref{update1}, suppose $\bbf$ is a  Lipschitz continuous vector field and $\bx(t)$ is a continuous $\mathbb{R}^d$-valued function. Then $\dot{\bx}(t)$ is continuous. Our aim is to obtain a good estimate of $\dot{\bx}(t)$ from the noisy time-series data $\{(t_i,\by_i)\}_{i=1}^{n}$. We assume that $\bphi(t):=\dot{\bx}(t)$ belongs to an $\mathbb{R}^{d}$-valued RKHS $\calH_K$ with reproducing kernel $K(t,s)$.

From now on, we consider the vRKHS directly related to our focused problem, i.e., we set $X=[0,T]$, but for notational simplicity, in some contexts we also use $X$ as an alternative. To estimate $\dot{\bx}(t)$, we consider the following regularization problem in $\calH_K$:
\begin{equation}\label{loss_1}
  \min_{\bphi\in\calH_K}\sum_{i=1}^{n}\left\|\bx_0+\int_{0}^{t_i}\bphi(s)\rd s - \by_i\right\|_{2}^2+\lambda\|\bphi\|_{\calH_K}^2 ,
\end{equation}
where the integral is calculated componentwisely for the $\mathbb{R}^{d}$-valued $\bphi(s)$. Note that for $\bphi=\dot{\bx}$, then $\bx_0+\int_{0}^{t_i}\bphi(s)\rd s=\bx(t_i)$ for $i=1,\dots,n$. Thus, \cref{loss_1} is a regularized loss function to fit the noisy data $\{(t_i,\by_i)\}_{i=1}^{n}$.

At first glance, \cref{loss_1} appears quite similar to \cref{regu2}. However, the Representer Theorem cannot be applied to \cref{loss_1}, as it does not include any $\bphi(t_i)$ term. In the following part, we establish a new representer theorem, which asserts that the solution to \cref{loss_1} can be expressed as a linear combination of certain basis functions.

For any $\bphi=(\phi_1,\dots,\phi_d)^{\top}\in\calH_K$ and $t_i\in[0,T]$, define the linear map 
\begin{equation}\label{int_op}
  \calL_i: \calH_K \rightarrow \mathbb{R}^{d}, \ \ \ 
  \bphi \mapsto \int_{0}^{t_i}\bphi(s)\rd s := \begin{pmatrix}\int_{0}^{t_i}\phi_{1}(s)\rd s, \cdots, \int_{0}^{t_i}\phi_{d}(s)\rd s\end{pmatrix}^{\top} .
\end{equation}
Define the feature map 
\begin{equation}
  \Phi: X\times \mathbb{R}^{d} \rightarrow \calH_K , \ \ \ 
  (x, v) \mapsto K(\cdot, x)v .
\end{equation}
For any two metric spaces $X$ and $Y$, denote by $C(X,Y)$ the set of all continuous maps from $X$ to $Y$. If $K\in C(X\times X, \mathbb{R}^{d\times d})$, we call $K$ a continuous kernel. The following result will be used to establish the new representer theorem.

\begin{lemma}\label{lem:functional_reproducing}
  Let $\calH_K$ be an $\mathbb{R}^d$-valued RKHS with a continuous kernel $K$. For any $\bphi\in\calH_K$ and $v\in\mathbb{R}^{d}$ and any $\calL_i$ defined as \cref{int_op}, it holds
  \begin{equation}\label{L_func}
    \langle \calL_i(\bphi), v\rangle_2 = \langle \bphi, \calL_i(\Phi(\cdot,v)) \rangle_{\calH_K} .
  \end{equation}
\end{lemma}
\begin{proof}
  Since $X$ is compact, it follows that the function $x\mapsto\|K(x,x)\|_2$ is locally bounded and $x\mapsto K(x,x')v \in C(X,\mathbb{R}^{d})$ for any fixed $x\in X$ and $v\in\mathbb{R}^{d}$. Therefore, $K$ is a Mercer kernel and $\calH_{K}$ is continuously embedded in $C([0, T],\mathbb{R}^{d})$; see \cite[Proposition 2.2]{carmeli2010vector}. We prove \cref{L_func} by the following two steps.

  Step 1: prove $\calL_i(\Phi(\cdot,v))\in\calH_K$. Fixing $t$ and $v$, then $K(s,t)v$ is a continuous map with values in $\mathbb{R}^d$. Thus, there exist countably many positive numbers $w_j$ and $s_j\in[0,t_i]$ such that
  \begin{align*}
    \calL_i(\Phi(\cdot,v))(t)=\int_{0}^{t_i}K(s,t)v\rd s
    = \lim_{m\rightarrow\infty}\sum_{j=1}^{m}w_j K(s_j,t)v
    =: \lim_{m\rightarrow\infty}\bphi_{i,m}(t),
  \end{align*}
  for any $t\in[0,T]$, and $\bphi_{i,m}=\sum_{j=1}^{m}w_j K(s_j,\cdot)v\in\calH_K$. Note that the integral $\int_{0}^{t_i}K(s,t)v\rd s$ is calculated componentwisely for $K(s,t)v$.  
  Since $[0,T]$ is compact and $\bphi_{i,m}$ is continuous, the above convergence is uniform and $\bphi_{i,m}\rightarrow\calL_i(\Phi(\cdot,y))$ in the Banach space $C([0,T],\mathbb{R}^{d})$. Now we prove $\{\bphi_{i,m}\}$ is a Cauchy sequence in $\calH_K$. For any $m'>m$, we have
  \begin{align*}
    \|\bphi_{i,m'}-\bphi_{i,m}\|_{\calH_K}^{2}
    &= \left\langle \sum_{j=m+1}^{m'}w_j K(s_j,\cdot)v, \sum_{l=m+1}^{m'}w_l K(s_l,\cdot)v \right\rangle_{\calH_K} \\
    &= v^{\top}\left(\sum_{j,l=m+1}^{m'}w_jw_l K(s_j,s_l)\right)v \rightarrow 0 \ \ (m\rightarrow +\infty),
  \end{align*}
  since 
  \begin{equation*}
    \sum_{j,l=1}^{m}w_jw_l K(s_j,s_l) \rightarrow \iint_{[0,t_i]^2}K(s,t)\rd s \rd t < \infty \ \ (m\rightarrow +\infty),
  \end{equation*}
  where the above integral and ``$<$'' are processed componentwisely.
  Therefore, there exist a $\bar{\bphi}\in\calH_K$ such that $\bphi_{i,m}\rightarrow\bar{\bphi}$ with the $\calH_K$ norm. Since the $\calH_K$ norm is stronger than the $C([0,T],\mathbb{R}^{d})$ norm, thereby $\bphi_{i,m}\rightarrow\bar{\bphi}$ with the $C([0,T],\mathbb{R}^{d})$ norm, which leads to $\calL_i(\Phi(\cdot,y))=\bar{\bphi}\in\calH_K$.

  Step 2: prove the equality \cref{L_func}. Using the above convergence relation, we have
  \begin{align*}
    \langle \bphi, \calL_i(\Phi(\cdot,v)) \rangle_{\calH_K}
    &= \langle \bphi,\sum_{j=1}^{\infty}w_j K(s_j,\cdot)v \rangle_{\calH_K} \\
    &= \sum_{j=1}^{\infty} \langle \bphi, w_j K(s_j,\cdot)v \rangle_{\calH_K} 
    = \sum_{j=1}^{\infty} w_j \langle \bphi(s_j), v \rangle_2 \\
    &= v^{\top} \cdot \int_{0}^{t_i}\bphi(s)\rd s = \langle \calL_i(\bphi), v\rangle_2 .
  \end{align*}
  The proof is completed.
\end{proof}

Recall that for a Banach space, a functional $f: Z \rightarrow \mathbb{R}$ is lower-semicontinous if $\liminf_{n\rightarrow n}f(z_n)\geq f(z)$ for any convergent $z_n\rightarrow z$ in $Z$. The functional $f$ is coercive if $\lim_{\|x\|_Z\rightarrow +\infty}f(x)=+\infty$. Now we can give the representer theorem for \cref{loss_1}.

\begin{theorem}[Integral-form representer theorem]\label{repre_new}
  Let $\calH_K$ be an $\mathbb{R}^d$-valued RKHS with a continous kernel $K$. Suppose we have data pairs $(t_1,\tilde{\by}_1),\dots,(t_n,\tilde{\by}_n)\in X\times\mathbb{R}^d$, linear operators $\calL_1,\dots,\calL_N:\calH_K\rightarrow\mathbb{R}^{d}$ defined as \cref{int_op}, and $\lambda>0$. Then the regularization problem 
  \begin{equation}\label{vrkhs_regu}
    \min_{\bphi\in\calH_K} J(\bphi):= \sum_{i=1}^{n}\|\calL_i(\bphi)-\tilde{\by}_i\|_{2}^2 + \lambda\|\bphi\|_{\calH_K}^2
  \end{equation}
  has at least one solution, and any solution must have the representation
  \begin{equation}\label{expan_new}
    \bphi^{*} = \sum_{i=1}^{n}\calL_i(\Phi(\cdot,v_i))
  \end{equation}
  with some $v_i\in\mathbb{R}^d$ for $i=1,\dots,n$.
\end{theorem}
\begin{proof}
  Note that $J(\bphi)$ is bounded from below, coercive and lower-semicontinuous on $\calH_K$. It follows that $J(\bphi)$ has at least one minimizer. Suppose that $\bphi^*(x)$ is such a minimizer. Define the space 
  \begin{equation}
    \mathcal{M} = \mathrm{span}\{\bphi=\calL_i(\Phi(\cdot,v)), \ 1\leq i \leq N, \ v\in\mathbb{R}^{d}\}.
  \end{equation}
  It follows from \Cref{lem:functional_reproducing} that $\mathcal{M}$ is a finite-dimensional subspace of $\calH_K$. Using \Cref{lem:functional_reproducing} again, we get
  \begin{align*}
    \|\calL_i(\bphi)-\tilde{\by}_i\|_{2}^2 
    &= \langle \calL_i(\bphi), \calL_i(\bphi) \rangle_{2} - 2\langle \calL_i(\bphi), \tilde{\by}_i \rangle_{2} + \|\tilde{\by}_i\|_{2}^{2} \\
    &= \langle \bphi, \calL_{i}(\Phi(\cdot,\calL_i(\bphi)-2\tilde{\by}_i)) \rangle_{\calH_K} + \|\tilde{\by}_i\|_{2}^{2}.
  \end{align*}
  Denote by the $\mathcal{P}_{\mathcal{M}}$ the projector operator onto the closed subspace $\mathcal{M}$. Then we have
  \begin{align*}
    & \ \ \ \ \langle \bphi, \calL_{i}(\Phi(\cdot,\calL_i(\bphi)-2\tilde{\by}_i)) \rangle_{\calH_K} \\
    &= \langle \mathcal{P}_{\mathcal{M}}\bphi, \calL_{i}(\Phi(\cdot,\calL_i(\bphi)-2\tilde{\by}_i)) \rangle_{\calH_K} + \langle (I-\mathcal{P}_{\mathcal{M}})\bphi, \calL_{i}(\Phi(\cdot,\calL_i(\bphi)-2\tilde{\by}_i)) \rangle_{\calH_K} \\
    &= \langle \mathcal{P}_{\mathcal{M}}\bphi, \calL_{i}(\Phi(\cdot,\calL_i(\bphi)-2\tilde{\by}_i)) \rangle_{\calH_K} \\
    &= \langle \calL_i(\mathcal{P}_{\mathcal{M}}\bphi), \calL_i(\bphi) \rangle_{2} - 2\langle \calL_i(\mathcal{P}_{\mathcal{M}}\bphi), \tilde{\by}_i \rangle_{2},
  \end{align*}
  since $\calL_{i}(\Phi(\cdot,\calL_i(\bphi)-2\tilde{\by}_i))\in\mathcal{M}$. Using the same approach as the above, we can also get $\langle \calL_i(\mathcal{P}_{\mathcal{M}}\bphi), \calL_i(\bphi) \rangle_{2}=\langle \calL_i(\mathcal{P}_{\mathcal{M}}\bphi), \calL_i(\mathcal{P}_{\mathcal{M}}\bphi) \rangle_{2}$. Therefore, we obtain 
  \begin{equation*}
    \|\calL_i(\bphi)-\tilde{\by}_i\|_{2}^2  = \|\calL_i(\mathcal{P}_{\mathcal{M}}\bphi)-\tilde{\by}_i\|_{2}^2 .
  \end{equation*}
  If $\bphi^{*}\notin\mathcal{M}$, using the above equality, we have
  \begin{align*}
    J(\bphi^{*}) 
    &= \sum_{i=1}^{N}\|\calL_i(\mathcal{P}_{\mathcal{M}}\bphi^{*})-\tilde{\by}_i\|_{2}^2 + \lambda\|\mathcal{P}_{\mathcal{M}}\bphi^{*}+(I-\mathcal{P}_{\mathcal{M}})\bphi^{*}\|_{\calH_K}^2 \\
    &= J(\mathcal{P}_{\mathcal{M}}\bphi^{*}) + \lambda\|(I-\mathcal{P}_{\mathcal{M}})\bphi^{*}\|_{\calH_K}^2 \\
    &> J(\mathcal{P}_{\mathcal{M}}\bphi^{*}) ,
  \end{align*}
  contradictory with that $\bphi^{*}$ is a minimizer. In other words, $\bphi^{*}$ must have the finite expansion form \cref{expan_new}.
\end{proof}

Note that \cref{loss_1} and \cref{vrkhs_regu} are identical by setting $\tilde{\by}_i=\by_i-\bx_0$. The Integral-form Representer Theorem (IRT) allows us to get a finite-dimensional optimization problem from \cref{loss_1}. We can further show that the solution of \cref{loss_1} is unique.

\begin{corollary}
  The regularization problem \cref{loss_1} has a unique solution, which has the expression \cref{expan_new} with $\bv=(v_{1}^{\top},\dots,v_{n}^{\top})^{\top}\in\mathbb{R}^{nd}$ be an arbitrary solution of 
  \begin{equation}\label{ls_1}
  \min_{\bv\in\mathbb{R}^{nd}} \|\mathbf{G}\mathbf\bv-\bb\|_{2}^{2} + \lambda \bv^{\top}\mathbf{G}\bv ,
  \end{equation}
  where 
  \begin{equation}\label{mat_vec}
  \mathbf{G}=\begin{pmatrix}
    G_{11} & \cdots & G_{1n} \\
    \vdots & \ddots & \vdots \\
    G_{n1} & \cdots & G_{nn}
  \end{pmatrix}, \quad
  \bb = \begin{pmatrix}
    b_1 \\ \vdots \\ b_n
  \end{pmatrix} 
  \end{equation}
  with $b_i=\by_i-\bx_0$ and $G_{ij}=\int_{0}^{t_i}\int_{0}^{t_j}K(s,t)\rd s\rd t$.
\end{corollary}
\begin{proof}
  By the IRT, suppose $\bphi = \sum_{i=1}^{n}\calL_i(\Phi(\cdot,v_i))$ be a minimizer of \cref{loss_1}. Then 
\begin{align*}
  \bx_0+\int_{0}^{t_i}\bphi(s)\rd s - \by_i
  &= \int_{0}^{t_i}\left(\sum_{j=1}^{n}\int_{0}^{t_j}K(s,t)v_j\rd t \right)\rd s - (\by_i-\bx_0) \\
  &= \sum_{j=1}^{n}\int_{0}^{t_i}\int_{0}^{t_j}K(s,t)\rd s\rd t \cdot v_j - (\by_i-\bx_0) .
\end{align*}
Let $\mathbf{G}$ and $\bb$ be defined in \cref{mat_G1}. Then we have
\begin{equation}
  \sum_{i=1}^{n}\|\bx_0+\int_{0}^{t_i}\bphi(s)\rd s - \by_i\|_{2}^2
  = \|\mathbf{G}\bv-\bb\|_{2}^{2} .
\end{equation}
Now we derive the expression of $\|\bphi\|_{\calH_K}^2$. Using \Cref{lem:functional_reproducing} we have
\begin{align*}
  \|\bphi\|_{\calH_K}^2
  &= \left\langle \sum_{i=1}^{n}\calL_i(\Phi(\cdot,v_i)), \sum_{j=1}^{n}\calL_j(\Phi(\cdot,v_j)) \right\rangle_{\calH_K} \\
  &= \sum_{i,j=1}^{n}\langle \calL_i\calL_j(\Phi(\cdot,v_j)),v_i \rangle_{2} \\
  &= \sum_{i,j=1}^{n} v_{i}^{\top} \left(\int_{0}^{t_i}\int_{0}^{t_j}K(s,t)\rd s\rd t \right) v_{j} \\
  & = \bv^{\top}\mathbf{G}\bv .
\end{align*}
Therefore, the coefficient vector must be a solution of the least squares problem \cref{ls_1}.

Note that there may be multiple solutions of \cref{ls_1} if $\mathrm{rank}(G)<nd$. Now we prove that even for this case, these different minimizers will lead to the same solution of \cref{loss_1}. Suppose $\bv$ and $\bv'=(v_{1}'^{\top},\dots,v_{n}'^{\top})^{\top}\in\mathbb{R}^{nd}$ are two solutions of \cref{ls_1} and the corresponding solutions to \cref{loss_1} are $\bphi$ and $\bphi'=\sum_{i=1}^{n}\calL_i(\Phi(\cdot,v_{i}'))$, respectively. Then it holds that $\bw:=\bv-\bv'\in\mathcal{N}(\mathbf{G})$, the null space of $\mathbf{G}$. Note that 
\begin{equation*}
  \bphi-\bphi'= \sum_{i=1}^{n}(\calL_i(\Phi(\cdot,v_{i}))-\calL_i(\Phi(\cdot,v_{i}')))
  =\sum_{i=1}^{n}\calL_i(\Phi(\cdot,v_{i}-v_{i}')).
\end{equation*}
Using the same procedure as the above for deriving the expression of $\|\bphi\|_{\calH_K}^2$, we have
\begin{equation*}
  \|\bphi-\bphi'\|_{\calH_K}^2 = (\bv-\bv')^{\top}\mathbf{G}(\bv-\bv')=0.
\end{equation*}
Therefore, it must hold that $\bphi=\bphi'$.
\end{proof}

In the next section, we will discuss the choice of the kernel $K$ for practical computations. With an appropriate selection of $K$, a unique solution to \cref{ls_1} exists, and we will also present an efficient method for estimating the optimal value of $\lambda$.

\section{Practical computations for fitting derivative and trajectory simultaneously}\label{sec4}
We show how to choose an appropriate kernel $K$ to avoid the need for integral computations when forming $\mathbf{G}$. Additionally, we provide an efficient approach for determining an optimal regularization parameter $\lambda$ without solving \cref{ls_1} multiple times.

\subsection{Choice of kernels}
In practical computation, the most common choice of an $\mathbb{R}^{d}$-valued kernel can be a separable kernel $K(s,t)=k(s,t)A$, where $k(s,t)\in\mathbb{R}$ is a symmetric positive definite kernel function and $A\in\mathbb{R}^{d\times d}$ is a symmetric positive definite matrix. Without loss of generality, it can also be written as $K(s,t)=k(s,t)I_{d}$ since $A$ can be ``absorbed'' in the coefficients $\{v_i\}$. 
The following result describes the structure of an $\mathbb{R}^{d}$-valued RKHS with a separable kernel.

\begin{proposition}\label{prop:sepa_kernel}
  Let $X=[0,T]$. Let $k(t,t'):X\times X\rightarrow\mathbb{R}$ be a real-valued reproducing kernel with the corresponding RKHS $\calH_k$. Then $K(t,t'):=k(t,t')I_d:X\times X\rightarrow\mathbb{R}^{d}$ is an $\mathbb{R}^{d}$-valued reproducing kernel, and the corresponding $\mathbb{R}^{d}$-valued RKHS is 
\begin{equation}
  \calH_K = \underbrace{H_k \otimes \cdots \otimes H_k}_{d},
\end{equation}
where the inner product is $\langle \boldsymbol{g}, \boldsymbol{g}' \rangle_{\calH_K}=\sum_{i=1}^{d}\langle g_i, g_{i}' \rangle_{\calH_k}$ for any $\boldsymbol{g}=(g_{1},\dots,g_{d})^{\top}$ and $\boldsymbol{g}'=(g_{1}',\dots,g_{d}')^{\top}$.
\end{proposition}
\begin{proof}
  Write $\calH=\underbrace{H_k \otimes \cdots \otimes H_k}_{d}$ with the inner product defined as the above. Then $\calH$ is a Hilbert space. It is easy to verify that $K$ is a reproducing kernel. By the Moore-Aronszajn Theorem, we only need to verify that $\calH$ and $K$ satisfy the two properties in \cref{def:vrkhs}:
  \begin{enumerate}
    \item[(1).] For a fixed $t'\in X$ and $v\in\mathbb{R}^{d}$, consider the map $\boldsymbol{\phi}(t)=K(t,t')v=k(t,t')\boldsymbol{v}$ where $t\in X$ and $\boldsymbol{v}=(v_1,\dots,v_d)^{\top}\in\mathbb{R}^{d}$. If follows that $\boldsymbol{\phi}(t)=(k(t,t')v_1,\dots,k(t,t')v_d)^{\top}$. Since $k(t,t')v_i\in\calH_{k}$ for $i=1,\dots,d$, we have $\boldsymbol{\phi}(t)\in\calH$.
    \item[(2).] For any $t\in X$, $v\in\mathbb{R}^{d}$ and $\boldsymbol{\phi}\in\calH$, we have
    \begin{align*}
      \langle \boldsymbol{\phi}, K(t,\cdot)v\rangle_{\calH} =
      \langle \boldsymbol{\phi}, k(t,\cdot)v\rangle_{\calH}
      =\sum_{i=1}^{d} \langle \phi_i, k(t,\cdot)v_i\rangle_{\calH} 
      = \sum_{i=1}^{d} \phi_i v_i
      = \langle\boldsymbol{\phi}(x),v\rangle_2.
    \end{align*}
  \end{enumerate}
  Therefore, we get $\calH_K=\calH$.
\end{proof}

The following result shows that if $k(t,t')$ is a strictly positive definite kernel, then we can make sure that \cref{ls_1} must have a unique solution. For more discussions about strictly positive definite kernels, see \cite[Chapter 6]{wendland2004scattered}.
\begin{theorem}\label{thm:spd}
  Following the notations in \Cref{prop:sepa_kernel}, assume that $k(t,t')$ is a continuous and strictly positive definite kernel, which means that 
  \begin{equation*}
    \sum_{i,j=1}^{n}\alpha_{i}\alpha_jk(t_i,t_j) > 0
  \end{equation*}
  for any integer $n\geq 1$ and $(\alpha_1,\dots,\alpha_{n})\neq \mathbf{0}$. Then the matrix $\mathbf{G}\in\mathbb{R}^{nd\times nd}$ in \cref{mat_vec} is symmetric positive definite.
\end{theorem}
\begin{proof}
  It is easy to verify that $\mathbf{G}$ is symmetric, because
  \begin{equation*}
    G_{ji}^{\top} = \int_{0}^{t_j}\int_{0}^{t_i}K(s,t)^{\top}\rd s\rd t
    = \int_{0}^{t_i}\int_{0}^{t_j}K(t,s)\rd t \rd s
    =G_{ij}.
  \end{equation*}
  From the derivation of the expression of $\|\bphi\|_{\calH_K}^2$, we know that for $\bv^{\top}\mathbf{G}\bv\geq 0$ for any $\bv\in\mathbb{R}^{nd}$, which implies that $\mathbf{G}$ is symmetric positive semi-definite. 
  
  To prove the strict positiveness, suppose $\bv^{\top}\mathbf{G}\bv=0$ for some $\bv\in\mathbb{R}^{nd}$ and write $\bphi:= \sum_{i=1}^{n}\calL_i(\Phi(\cdot,v_i))$ where $\bv=(v_{1}^{\top},\dots,v_{n}^{\top})^{\top}$ with $v_{i}\in\mathbb{R}^{d}$. Then $\|\bphi\|_{\calH_K}^2=0$, leading to $\bphi = \mathbf{0}$, that is
\begin{equation}\label{equal1}
  \mathbf{0}=\sum_{i=1}^{n}\int_{0}^{t_i}K(s,\cdot)\rd s \cdot v_i = \sum_{i=1}^{n}\int_{0}^{t_i}k(s,\cdot)\rd s \cdot v_i .
\end{equation}
Note that $\int_{0}^{t_i}k(s,\cdot)$ is a scalar-valued function. Now we prove that $\{\int_{0}^{t_i}k(s,\cdot)\}_{i=1}^{n}$ are linear independent functions. Suppose there exist real numbers $c_{1},\dots,c_{n}$ such that 
\begin{equation*}
  0 = \sum_{i=1}^{n}c_i\int_{0}^{t_i}k(s,\cdot)\rd s = \sum_{i=1}^{n}\sum_{j=1}^{i}\int_{t_{j-1}}^{t_j}k(s,\cdot)\rd s
  = \sum_{j=1}^{n}\left(\sum_{i=j}^{n}c_{i}\right)\int_{t_{j-1}}^{t_j}k(s,\cdot)\rd s,
\end{equation*}
where $t_{0}=0$. By the Mean Value Theorem for Integrals, there exist $\xi_{i}\in(t_{j-1},t_{j})$ for $j=1,\dots,n$ such that $\int_{t_{j-1}}^{t_j}k(s,\cdot)\rd s=(t_{j}-t_{j-1})k(\xi_{j},\cdot)$. Since $k(\cdot,\cdot)$ is strictly positive definite, it follows that the matrix $(k(\xi_i,\xi_{j}))_{1\leq i,j\leq n}$ is positive definite, and thereby the $n$ vectors $\{(k(\xi_{i},\xi_j))_{1\leq j\leq n}\}_{i=1}^{n}$ are linear independent. Using \cite[Theorem 2.1]{wu2024characterizations}, it follows that $\{k(\xi_{i},\cdot)\}_{i=1}^{n}$ are linear independent functions. Therefore, we have $\sum_{i=j}^{n}c_{i}=0$ for any $1\leq j\leq n$, i.e.
\begin{equation*}
  \mathbf{0}= (c_1,\dots,c_{n})\begin{pmatrix}
    1 & &  \\
    \vdots & \ddots & \\
    1  & \cdots & 1
  \end{pmatrix} ,
\end{equation*}
leading to $c_1=\cdots=c_n=0$. This proves that $\{\int_{0}^{t_i}k(s,\cdot)\rd s\}_{i=1}^{n}$ are linear independent. Using \cite[Theorem 2.1]{wu2024characterizations} again, we can find $n$ difference points $\{s_{i}\}_{i=1}^{n}\subset [0,T]$ such that the $n$ vectors $\{(\int_{0}^{t_i}k(s,s_j)\rd s)_{1\leq j\leq n}\}_{i=1}^{n}$ are linear independent. Now \cref{equal1} implies that
\begin{equation*}
  \mathbf{0} = (v_{1},\cdots,v_{n})\begin{pmatrix}
    \int_{0}^{t_1}k(s,s_1)\rd s & \cdots & \int_{0}^{t_1}k(s,s_n)\rd s \\
    \vdots & \cdots & \vdots \\
    \int_{0}^{t_n}k(s,s_1)\rd s & \cdots & \int_{0}^{t_n}k(s,s_n)\rd s
  \end{pmatrix}.
\end{equation*}
Since the above matrix is nonsingular, it follows that $v_{i}=\mathbf{0}$ for $1\leq i\leq n$. This proves that $\mathbf{G}$ is strictly positive define.
\end{proof}

By \Cref{thm:spd}, if we choose a separable kernel $K(s,t)=k(s,t)I_{d}$ with a strictly positive definite $k(t,t')$, then $\mathbf{G}$ is positive definite, thereby the unique minimizer of \cref{ls_1} is
\begin{equation*}
  \bv_{\lambda}=(\mathbf{G}^{2}+\mathbf{G})^{-1}\mathbf{G}\bb = (\mathbf{G}+\lambda I_{nd})^{-1}\bb .
\end{equation*}
Moreover, $\mathbf{G}$ is a block diagonal matrix, which has the form
\begin{equation*}
  \mathbf{G}=G_{1}\otimes I_d
\end{equation*}
with
\begin{equation}\label{mat_G1}
  G_1=\begin{pmatrix}
    \int_{0}^{t_1}\int_{0}^{t_1}k(s,t)\rd s\rd t & \cdots & \int_{0}^{t_1}\int_{0}^{t_n}k(s,t)\rd s\rd t \\
    \vdots & \ddots & \vdots \\
    \int_{0}^{t_n}\int_{0}^{t_1}k(s,t)\rd s\rd t & \cdots & \int_{0}^{t_n}\int_{0}^{t_n}k(s,t)\rd s\rd t
  \end{pmatrix} .
\end{equation}
In this case, we have
\[(\mathbf{G}+\lambda I_{nd})^{-1} = [G_1\otimes I_d + \lambda I_n \otimes I_{d}]^{-1} = 
[(G_{1}+\lambda I_{n})\otimes I_{d}]^{-1} = (G_{1}+\lambda I_{n})^{-1} \otimes I_{d} . \]
Let $\mathbf{B}=(b_1,\dots,b_n)\in\mathbb{R}^{d\times n}$. Then we have 
\begin{equation}\label{vec_v}
  \bv_{\lambda} = [(G_{1}+\lambda I_{n})^{-1} \otimes I_d] \bb 
  = \rvec(\mathbf{B}(G_{1}+\lambda I_{n})^{-1}),
\end{equation}
where $\mathrm{vec}(\cdot)$ vectorizes a matrix by stacking its columns from left to right into a single column vector. Write $\bv_{\lambda}=(v_{\lambda,1}^{\top},\dots,v_{\lambda,n}^{\top})^{\top}$. We eventually obtain the solution of \cref{loss_1}, which is the estimated derivative function from the noisy data:
\begin{equation}\label{est_deriv}
  \dot{\bx}(t) \approx \bphi^{*}(t) = \sum_{j=1}^{n}\int_{0}^{t_j}K(s,t)v_{\lambda,j}\rd s
  = \sum_{j=1}^{n}\left(\int_{0}^{t_j}k(s,t)\rd s \right) v_{\lambda,j} .
\end{equation}

To estimate the values of $\dot{\bx}(t)$ at $\{t_{i}\}_{i=1}^{n}$, define $\psi_{j}(t_i)=\int_{0}^{t_j}k(s,t)\rd s$. It follows from \cref{est_deriv} that $\dot{\bx}(t_i)\approx\sum_{j=1}^{n}\psi_{j}(t_i)v_j$, and thereby
\begin{equation}\label{deriv_approx}
  \begin{pmatrix}
    \dot{\bx}(t_1) & \cdots  & \dot{\bx}(t_n)
  \end{pmatrix} \approx
  \begin{pmatrix}
    v_{\lambda,1} & \cdots & v_{\lambda,n}
  \end{pmatrix}
  \begin{pmatrix}
    \psi_{1}(t_1) & \cdots & \psi_{1}(t_n) \\
    \vdots & \ddots & \vdots \\
    \psi_{n}(t_1) & \cdots & \psi_{n}(t_n)
  \end{pmatrix} =: \mathbf{V}_{\lambda}\Psi ,
\end{equation}
where it follows from \cref{vec_v} that $\mathbf{V}_{\lambda}=\mathbf{B}(G_{1}+\lambda I_{n})^{-1}$. 
To estimate the values of $\bx(t)$ at $\{t_{i}\}_{i=1}^{n}$, notice that 
\begin{align*}
  \bx(t_i) &= \bx_0 + \int_{0}^{t_i}\dot{\bx}(t) \rd t \\
  &\approx \bx_0 + \int_{0}^{t_i}\left(\sum_{j=1}^{n}\int_{0}^{t_j}k(s,t)v_{\lambda,j}\rd s \right)\rd t \\
  &=  \bx_0 + \sum_{j=1}^{n}\left(\int_{0}^{t_i}\int_{0}^{t_j}k(s,t)\rd s\rd t \right) v_{\lambda,j} .
\end{align*}
Therefore, we obtain
\begin{equation}\label{traj_approx}
  \begin{pmatrix}
    \bx(t_1) & \cdots  & \bx(t_n)
  \end{pmatrix} \approx 
  \begin{pmatrix}
    \bx_0 & \cdots & \bx_0
  \end{pmatrix} + V_{\lambda}G_1 .
\end{equation}

During the computation, in order to form the Gram matrix $\mathbf{G}$ and basis functions $\psi_{j}(t)$, more than $n^2$ integrals with integrand $k(t,t')$ should be computed. However, by selecting proper kernel functions, analytical expressions for these integrals can be obtained, allowing the costly numerical integrations to be avoided. For the kernel function $k(s,t)$, it should be efficient to compute $\int_{0}^{t_i}\int_{0}^{t_j}k(s,t)\rd s\rd t$ and $\psi_{j}(t_i)=\int_{0}^{t_j}k(s,t)\rd s$. We exploit two types of commonly used kernel functions---the Gaussian kernel and the Mat\'{e}rn kernel. Both of them are strictly positive definite kernels; see \cite[Theorem 6.11]{wendland2004scattered} for how to check whether it is a strictly positive definite kernel.

The Gaussian kernel function
\begin{equation}
  k(s,t) = \exp\left(-\frac{(s-t)^2}{2l^2}\right)
\end{equation}
with $l$ a hyperparameter is a very common choice for fitting continuous curves; for the structure of the RKHS induced by the Gaussian kernel, see \cite[Theorem 5.20]{wendland2004scattered}. The parameter $l$ determines the length scale of the associated hypothesis space of functions. As $l$ increases, the induced functions change less rapidly, and thus get ``smoother". Although the analytic form of the antiderivative of this function can not be expressed with elementary functions, we can still quickly evaluate the value of its integral with the help of special functions. Using the error function
\[\mathrm{erf}(x) = \frac{2}{\sqrt{\pi}}\int_{0}^{x}e^{-t^2}\rd t ,\ \ \ x\in\mathbb{R}  \]
and its antiderivative 
\[h(x):=\int \mathrm{erf}(x)\rd x = x\cdot\mathrm{erf}(x)+\frac{e^{-x^2}}{\sqrt{\pi}} + c , \]
we have 
\begin{equation*}
  \int_{0}^{t_j}k(s,t)\rd s = 
  \frac{\sqrt{2\pi}l}{2}\left(\mathrm{erf}\left(\frac{t_j-t}{\sqrt{2}l}\right)+\mathrm{erf}\left(\frac{t}{\sqrt{2}l}\right) \right)
\end{equation*}
and 
\begin{equation*}
  \int_{0}^{t_i}\int_{0}^{t_j}k(s,t)\rd s\rd t
  = \sqrt{\pi}l^2 \left( h\left(\frac{t_i}{\sqrt{2}l}\right)+h\left(\frac{t_j}{\sqrt{2}l}\right)-h\left(\frac{t_i-t_j}{\sqrt{2}l}\right)-h(0) \right) .
\end{equation*}
Therefore, the numerical integrals for forming $G_1$ and $\Phi$ can be avoided.

Sometimes the hypothesis space of the Gaussian kernel is too smooth for fitting the derivative function. In this case, we can choose the Mat\'{e}rn kernel
\begin{equation}
  k_{\nu,l}(s,t) = \frac{2^{1-\nu}}{\Gamma(\nu)}\left(\frac{\sqrt{2\nu}|s-t|}{l}\right)^{\nu}K_{\nu}\left(\frac{\sqrt{2\nu}|s-t|}{l} \right) ,
\end{equation}
where $\nu, l>0$ are hyperparameters, $\Gamma(\cdot)$ is the gamma function, and $K_{\nu}(\cdot)$ is the modified Bessel function of the second kind of order $\nu$. The parameter $\nu$ effectively controls the level of smoothness of the function. In fact, for any Sobolev space $H^{\tau}(X)$ with $\tau=\nu+\frac{1}{2}$ is a positive integer, then the RKHS $\calH_{k_{\nu,l}}$ is equivalent to $H^{\tau}(X)$. Therefore, for the separable kernel $K(s,t)=k_{\nu,l}(s,t)I_{d}$, the corresponding $\mathbb{R}^{d}$-valued RKHS is $
  \calH_{K} = \underbrace{H^{\tau}(X) \otimes \cdots \otimes H^{\tau}(X)}_{d}$,
meaning that each component of the trajectory belongs to $H^{\tau}(X)$. 
For the Mat\'{e}rn kernel with $\nu=\tau-\frac{1}{2}=(\tau-1)+\frac{1}{2}$, a well-known property is that $k_{\nu,l}(s,t)$ can be written as
\begin{equation*}
  k_{\nu,l}(s,t) = \exp\left(-\frac{\sqrt{2\nu}|s-t|}{l}\right)p_{\tau-1}\left(\frac{|s-t|}{l}\right),
\end{equation*}
where $p_{\tau-1}(\cdot)$ is a polynomial of degree $\tau-1$; see \cite[\S 4.2.1]{williams2006gaussian}. Therefore, $ k_{\nu,l}(s,t)$ is integrable with respect to both $s$ and $t$. To efficiently compute $\int_{0}^{t_i}\int_{0}^{t_j}k(s,t)\rd s\rd t$ and $\psi_{j}(t_i)=\int_{0}^{t_j}k(s,t)\rd s$, we can first derive the analytic forms of the above two functions and then evaluate them at the corresponding points.

\subsection{Determining regularization parameter}
Now we show how to compute a good estimation of the optimal regularization parameter $\lambda$. Heuristically, a good $\lambda$ should balance the loss term and regularization term in \cref{ls_1}. We estimate the optimal parameter $\lambda$ by the L-curve method \cite{hansen1999curve}. The idea is to plot the following parametrized curve in log-log scale:
\begin{equation*}
  l(\lambda) = (x(\lambda), y(\lambda)): = \left(\log(\|\mathbf{G}\bv_{\lambda}-\bb\|_{2}), \log((\bv_{\lambda}^{\top}\mathbf{G}\bv_{\lambda})^{\frac{1}{2}})\right),
\end{equation*}
which usually has a characteristic ``L'' shape, and the corner of $l(\lambda)$ corresponds to the point where further reduction in the residual comes only at the expense of a drastic increase in the regularization term. Therefore, a good estimate of $\lambda$ should maximize the (signed) curvature of $l(\lambda)$. Let the eigen-decomposition of $G_1$ be $G_{1}=U\Lambda U^{\top}$, where $U=(u_1,\dots,u_n)\in\mathbb{R}^{n\times n}$ and $\Lambda=\diag(\lambda_{1},\dots,\lambda_n)$ with $\lambda_1\geq\cdots\geq\lambda_{n}>0$. In practical computation, we restrict $\lambda$ in the spectral range of $G_{1}$, and compute 
\begin{equation}\label{estim_lamb}
  \lambda^{*} = \argmax_{\lambda_{n}\leq\lambda\leq\lambda_{1}} \kappa(\lambda):=
  \frac{x'y''-y'x''}{(x'^{2}+y'^{2})^{3/2}}
\end{equation}
as the optimal regularization parameter. Since 
\begin{equation}\label{V_lamb}
  \mathbf{V}_{\lambda} =\mathbf{B}(G_{1}+\lambda I_{n})^{-1}
  = \mathbf{B}U(\Lambda+\lambda I)^{-1}U^{\top},
\end{equation}
we can quickly obtain $\mathbf{V}_{\lambda}$ for different $\lambda$ once we have the eigen-decomposition of $G_{1}$. Then we get the analytic expressions of $x(\lambda)$ and $y(\lambda)$ via the equalities
\begin{align*}
  \|\mathbf{G}\bv_{\lambda}-\bb\|_{2}
  = \|(G_{1}\otimes I_d)\bv_{\lambda}-\bb\|_{2} 
  = \|\rvec(\mathbf{V}_{\lambda}G_1)-\rvec(\mathbf{B})\|_{2} 
  = \|\mathbf{V}_{\lambda}G_{1}-\mathbf{B}\|_{F} 
\end{align*}
and 
\begin{align*}
  \bv_{\lambda}^{\top}\mathbf{G}\bv_{\lambda}
  &= \sum_{i,j=1}^{n}v_{i}^{\top}(G_{ij}I_d)v_j
  = \sum_{i=1}^{n}v_{i}^{\top}\left(\sum_{j=1}^{n}G_{ij}v_j\right)
  = \sum_{i=1}^{n}v_{i}^{\top}(\mathbf{V}_{\lambda}G_{1}^{\top})_i \\
  &= \mathrm{trace}(\mathbf{V}_{\lambda}^{\top}\mathbf{V}_{\lambda}G_{1}) 
  = \mathrm{trace}(\mathbf{V}_{\lambda}G_{1}\mathbf{V}_{\lambda}^{\top}) .
\end{align*}

The algorithm for estimating the values of $\dot{\bx}(t)$ and $\bx(t)$ at $\{t_{i}\}_{i=1}^{n}$ is summarized in \Cref{alg1:denoise}.

\begin{algorithm}[htb]
	\caption{Fitting derivative and trajectory by vRKHS}\label{alg1:denoise}
 	\algorithmicrequire \ Initial value $\bx_0$, noisy observation $\{(t_i, \by_i)\}_{i=1}^{n}$; kernel function $k(t,t')$
	\begin{algorithmic}[1]
		\State Form matrix $G_1$ by \cref{mat_G1}
    \State Form matrix $\mathbf{B}=(\by_1-\bx_0,\dots,\by_n-\bx_0)$
    \State Compute the eigen-decomposition of $G_1$: $G_1=U\Lambda U^{\top}$
    \State Estimate the optimal $\lambda$ by \cref{estim_lamb}
    \State Compute $\mathbf{V}_{\lambda}$ by \cref{V_lamb}
    \State Compute the fitted derivative by \cref{deriv_approx}
    \State Compute the fitted trajectory by \cref{traj_approx}
	\end{algorithmic}
 	\algorithmicensure \ Estimated $\begin{pmatrix}
    \dot{\bx}(t_1) & \cdots  & \dot{\bx}(t_n)
  \end{pmatrix}$ and $\begin{pmatrix}
    \bx(t_1) & \cdots  & \bx(t_n)
  \end{pmatrix}$
\end{algorithm}

\section{Learning dynamics from the fitted derivative and trajectory}\label{sec5}
In addition to the SINDy framework, the dynamics \(\bbf\) can also be embedded within a vRKHS. This strategy is particularly well-motivated due to the strong connection between the regularity properties of a kernel and those of \(\bbf\). Specifically, one can select a vRKHS such that the existence and uniqueness of the corresponding ODE are guaranteed \cite{lahouel2024learning}. An additional advantage of this approach is the elimination of the need to select a dictionary of pre-defined basis functions.

Denote the estimated values of $\{\dot{\bx}(t_i)\}_{i=1}^{n}$ and $\{\bx(t_i)\}_{i=1}^{n}$ by $\{\dot{\bphi}(t_i)\}_{i=1}^{n}$ and $\{\bphi(t_i)\}_{i=1}^{n}$, respectively. Assume $\bbf$ belong to an $\mathbb{R}^{d}$-valued RKHS on $\mathbb{R}^{d}$. We recover the dynamics $\bbf$ by fitting the ODE \cref{update1} at the time points $\{t_i\}_{i=1}^{n}$, i.e., we compute $\bbf$ by solving the following regularization problem:
\begin{equation}
  \min_{\widetilde{\bbf}\in\calH_K} \sum_{i=1}^{n}\|\dot{\bphi}(t_i)-\widetilde{\bbf}(\bphi(t_i))\|_{2}^2 + \lambda\|\widetilde{\bbf}\|_{\calH_K}^{2}.
\end{equation}
By \Cref{thm:rt}, the solution of this problem has the representation
\begin{equation}\label{recons_f}
  \widehat{\bbf} = \sum_{i=1}^{n}K(\bphi(t_i),\cdot)v_i
\end{equation}
with some $v_i\in\mathbb{R}^d$ for $i=1,\dots,n$. Write the coefficient vector as $\bv=(v_{1}^{\top},\dots,v_{n}^{\top})^{\top}$. 
By a similar discussion as the previous section, it follows that $\bv$ is the solution to the finite-dimensional regularization problem
\begin{equation}\label{regu_f}
  \min_{\bv\in\mathbb{R}^{nd}}\|\widetilde{\mathbf{G}}\bv-\widetilde{\bb}\|_{2}^{2}+\lambda \bv^{\top}\widetilde{\mathbf{G}}\bv
\end{equation}
where 
\begin{equation}
  \widetilde{\mathbf{G}} = 
  \begin{pmatrix}
    K(\bphi(t_1),\bphi(t_1)) & \cdots & K(\bphi(t_1),\bphi(t_n)) \\
    \vdots & \ddots & \vdots \\
    K(\bphi(t_n),\bphi(t_1)) & \cdots & K(\bphi(t_n),\bphi(t_n))
  \end{pmatrix}, \ \ \
  \widetilde{\bb} = 
  \begin{pmatrix}
    \dot{\bphi}(t_1) \\ \vdots \\ \dot{\bphi}(t_n)
  \end{pmatrix} .
\end{equation}
Similar to the approach in the previous section, by choosing $K(t,t')=k(t,t')I_d$ with a strictly positive definite kernel $k(t,t')$, we can write $\widetilde{\mathbf{G}}$ as $\widetilde{\mathbf{G}}=\widetilde{G}_{1}\otimes I_d$ with $\widetilde{G}_1=(k(\bphi(t_i),\bphi(t_j)))\in\mathbb{R}^{n\times n}$. Therefore, we obtain
\begin{equation}\label{vec_fv}
  \bv = [(\widetilde{G}_{1}+\lambda I_{n})^{-1} \otimes I_{d}] \widetilde{\bb} 
  = \rvec(\widetilde{\mathbf{B}}(\widetilde{G}_{1}+\lambda I_{n})^{-1}),
\end{equation}
where $\widetilde{\mathbf{B}}=(\dot{\bphi}(t_1),\dots,\dot{\bphi}(t_n))$.
Similarly, we can use the L-curve method to select a good parameter $\lambda$ for \cref{regu_f}.

Finally, the whole procedure for learning dynamics from noise data $\{(t_i,\by_{i})\}_{i=1}^{n}$ is shown in \Cref{alg2:dynamics}.
\begin{algorithm}[htb]
	\caption{Learning dynamics by vRKHS}\label{alg2:dynamics}
 	\algorithmicrequire \ Initial value $\bx_0$, noisy observation $\{(t_i, \by_i)\}_{i=1}^{n}$; kernel function $k(\bx,\bx')$
	\begin{algorithmic}[1]
		\State Estimate the values of $\{\dot{\bx}(t_i)\}_{i=1}^{n}$ and $\{\bx(t_i)\}_{i=1}^{n}$ as $\{\dot{\bphi}(t_i)\}_{i=1}^{n}$ and $\{\bphi(t_i)\}_{i=1}^{n}$
    \State Form matrix $\widetilde{G}_1=(k(\bphi(t_i),\bphi(t_j)))$ and $\widetilde{\mathbf{B}}=(\dot{\bphi}(t_1),\dots,\dot{\bphi}(t_n))$
    \State Estimate the optimal $\lambda$ by L-curve
    \State Compute $\bv$ by \cref{vec_fv}
    \State Compute the recovered dynamics by \cref{recons_f}
	\end{algorithmic}
 	\algorithmicensure \ Recovered $\widehat{\bbf}$
\end{algorithm}

\section{Numerical experiments}\label{sec6}
We choose several typical autonomous dynamical systems to test the proposed methods. \Cref{sec61} presents results for estimating derivatives and trajectories from the noisy time-series data. \Cref{sec62} presents results for learning dynamics based on the fitted trajectories and derivatives. The codes are available at \url{https://github.com/hessianguo/ODELearning}. 

\paragraph*{Forced vibration of nonlinear pendulum}
This is a second-order ODE describing the motion of a simple pendulum. It has the form $\frac{\rd^{2}\theta}{\rd t^2}+\alpha\sin(\theta)=f(\theta)$, where $\theta$ is the angle from the vertical to the pendulum, and $\alpha=g/L$ with $g$ and $L$ the acceleration of gravity and the length of the pendulum, respectively. We assume the external force $f$ only depends on $\theta$. Let $x_1=\theta$ and $x_2=\dot{\theta}$. We have the following equivalent 2-dimensional (dim) ODE:
\begin{equation}\label{pendulum}
  \begin{cases}
    \frac{\rd x_1}{\rd t} = x_2 \\
    \frac{\rd x_2}{\rd t} = f(x_1)-\alpha\sin x_1 .
  \end{cases}
\end{equation}
In the experiment, we set $L=5$, $f(x)=\cos(e^x)$. The initial value is set as $\bx_{0}=(0,0)^{\top}$.

\paragraph*{Lotka--Volterra equation}
The Lotka--Volterra equation, also known as the Lotka-Volterra predator-prey model, is a pair of first-order nonlinear differential equations used to describe the dynamics of biological systems in which two species interact, one as a predator and the other as prey \cite{lotka1925elements,volterra1928variations}. The ODE model is as follows:
\begin{equation}\label{LV}
  \begin{cases}
    \frac{\rd x_1}{\rd t} = \alpha x_1 - \beta x_1x_2 \\
    \frac{\rd x_2}{\rd t} = \delta x_1x_2 - \gamma x_2 .
  \end{cases}
\end{equation}
In the experiment, we set $(\alpha,\beta,\gamma,\delta)==(0.7, 0.007, 1, 0.007)$. The initial value is set as $\bx_0=(70, 50)^{\top}$.

\paragraph*{SIR model} 
The SIR model \cite{kermack1927contribution} describes the evolution of an epidemic by the follows ODEs:
\begin{equation}\label{sir}
  \begin{cases}
    \frac{\rd S}{dt} = - \beta \frac{SI}{S+I+R} \\
    \frac{\rd I}{dt} = \beta \frac{SI}{S+I+R} - \gamma I \\
    \frac{\rd R}{dt} = \gamma I ,
  \end{cases}
\end{equation}
where $S$, $I$, and $R$ are the numbers of susceptible, infected and recovered individuals, respectively. The parameter $\beta$ is the infection rate and $\gamma$ is the rate at which individuals recover. In the experiment, we set $(\beta, \gamma)=(0.4, 0.04)$ and $(S(0), I(0), R(0))=(900, 10, 0)$. 

\paragraph*{Lorenz63 model} The Lorenz63 model was proposed by the American meteorologist E. Lorenz in 1963 for for describing atmospheric turbulence \cite{lorenz1963deterministic}. It is a simple mathematical system constituted by three ordinary differential equations:
\begin{equation}\label{lorenz63_2}
  \begin{cases}
    \frac{\rd x_1}{\rd t} = \sigma(x_2-x_1) \\
    \frac{\rd x_2}{\rd t} = x_1(\rho-x_3)-x_2 \\
    \frac{\rd x_3}{\rd t} = x_1x_2-\beta x_3 .
  \end{cases}
\end{equation}
In the experiment, we set $(\sigma, \rho, \beta)=(10, 28, 8/3)$ and $\bx_0=(1, 1, 1)^{\top}$, which will lead to chaotic solutions and in particular, the Lorenz attractor.

\paragraph*{Lorenz96 model}
The Lorenz96 model is a dynamical system formulated by Edward Lorenz in 1996 \cite{lorenz1996predictability}, which is a system of ordinary differential equations that describes a single scalar quantity as it evolves on a circular array of sites, undergoing forcing, dissipation, and rotation invariant advection. It is defined as follows. For $i=1,\dots,N$, the equation is
\begin{equation}\label{lorenz96}
  \frac{\rd x_i}{\rd t} = (x_{i+1}-x_{i-2})x_{i-1} - x_i + F, \ \ \ 1\leq i \leq N, 
\end{equation}
where it is assumed that $x_{-1}=x_{N-1},x_0=x_N$ and $x_{N+1}=x_1$. Here $x_i$ is the state of the system and $F$ is a forcing constant. Here we set $N=5$ and $F=8$, which is a common value known to cause chaotic behavior. The initial value is set as $\bx_0=(8.01, 8, 8, 8, 8)^{\top}$.

For the first four ODEs, the trajectories and the noisy observation data are shown in \Cref{fig:noisy_data1}. For Lorenz96, we plot the trajectories and observations corresponding to the $x_1$-$x_2$-$x_3$ and $x_4$-$x_5$ projections.

\begin{figure}[htbp]
	\centering
	\subfloat[]
	{\label{fig:2a}\includegraphics[width=0.48\textwidth]{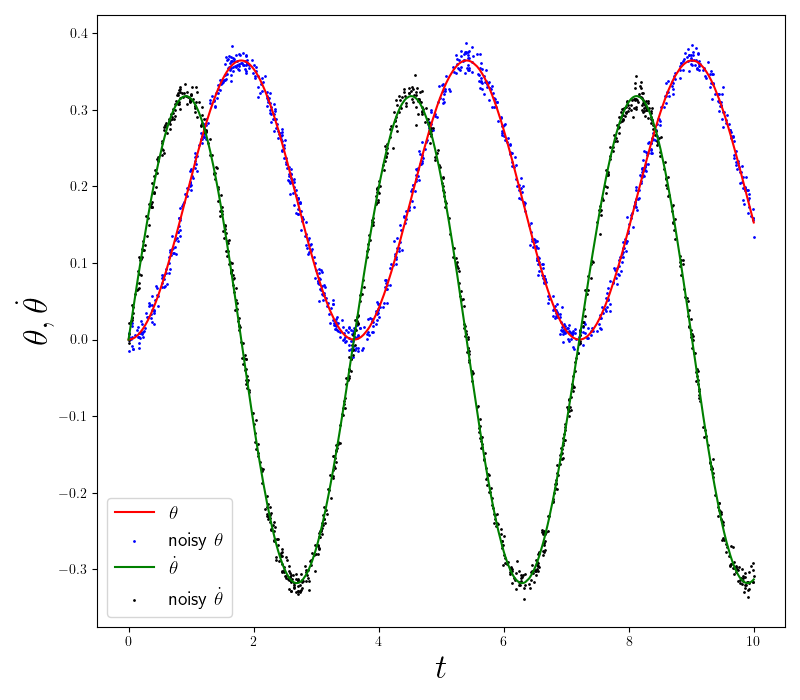}}\hspace{-1.0mm}
	\subfloat[]
	{\label{fig:2b}\includegraphics[width=0.48\textwidth]{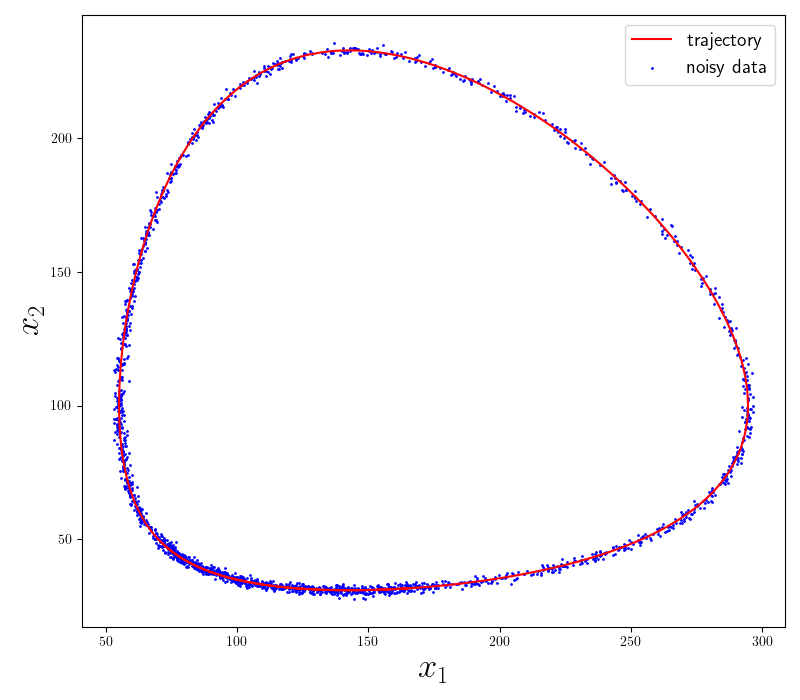}}
	\vspace{-3mm}
	\subfloat[]
	{\label{fig:2c}\includegraphics[width=0.46\textwidth]{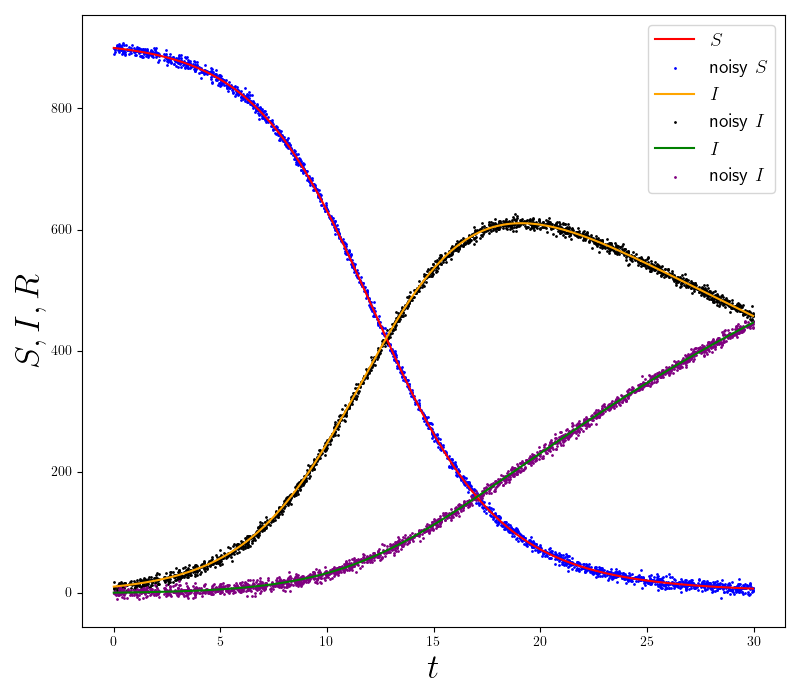}}\hspace{-2.0mm}
	\subfloat[]
	{\label{fig:2d}\includegraphics[width=0.5\textwidth]{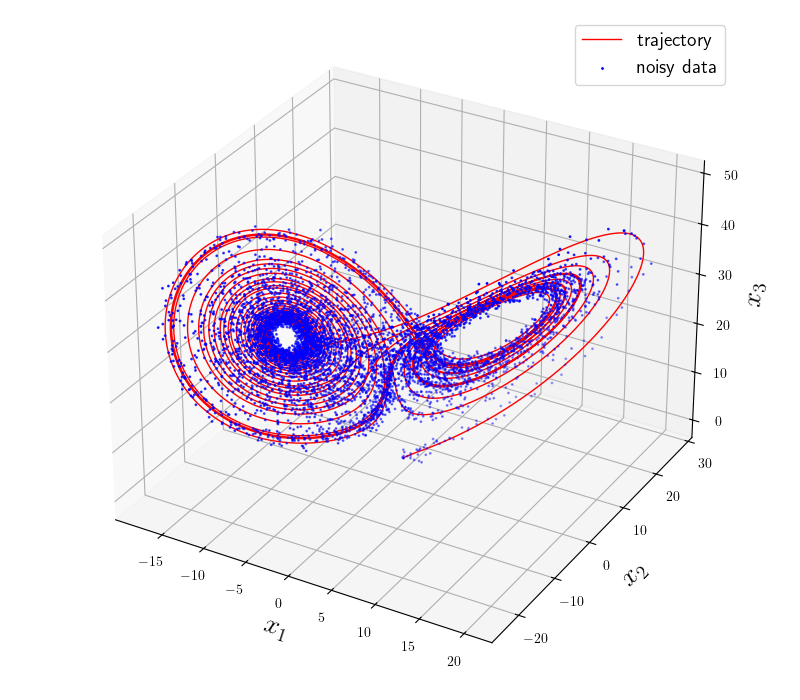}}
	\vspace{-3mm}
  \subfloat[]
	{\label{fig:2e}\includegraphics[width=0.5\textwidth]{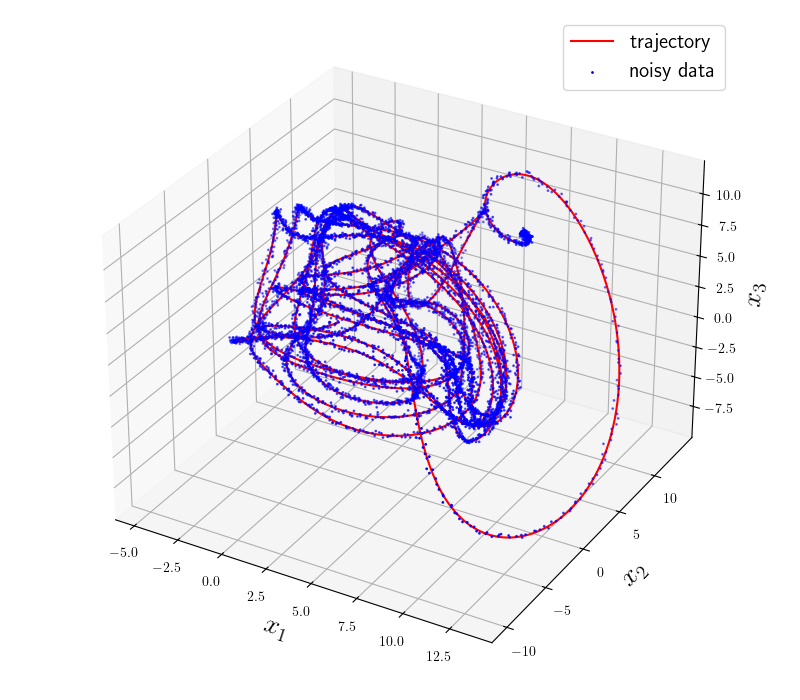}}\hspace{-2.0mm}
  \subfloat[]
	{\label{fig:2f}\includegraphics[width=0.45\textwidth]{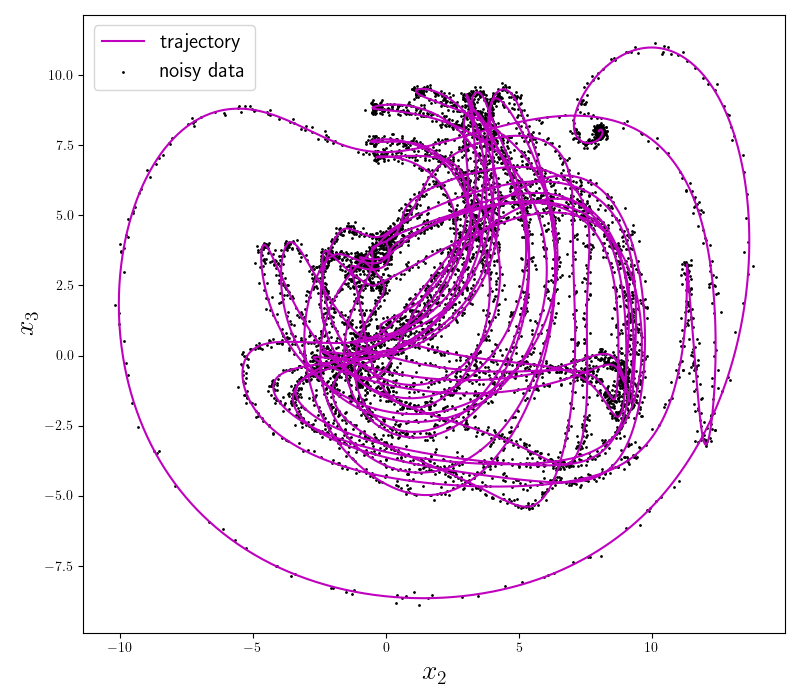}}
  \vspace{-2.0mm}
	\caption{Trajectories and noisy data. (a) Forced vibration of nonlinear pendulum, 1000 random times points in $[0,10]$ with noisy level $0.01$. (b) Lotka--Volterra equation, 2000 random times points in $[0,10]$ with noisy level $1.0$. (c) SIR model, 3000 random time points in $[0,30]$ with noisy level $5.0$. (d) Lorenz63 model, 6000 random times points in $[0,30]$ with noisy level $0.5$. (e), (f) Projections of the trajectories and noisy data for Lorenz96, 8000 random times points in $[0,30]$ with noisy level $0.1$.}
	\label{fig:noisy_data1}
\end{figure}

\subsection{Numerical comparison of the methods for estimating derivatives}\label{sec61}
In this subsection, we present a qualitative comparison study among the proposed RKHS method, finite difference methods, and the TV-regularized method \cite{chartrand2011numerical}. We note that there is no universally best method for estimating derivatives in all scenarios. However, our vRKHS method is particularly well suited for simultaneously fitting both derivatives and trajectories from noisy time-series data. 

In all the following computations, we choose the Gaussian kernel for estimating derivatives. For the nonlinear pendulum, Lotka--Volterra, SIR, Lorenz63, and Lorenz96 systems, the hyperparameter $\sigma$ is set to $0.2$, $0.4$, $5.0$, $0.04$, and $0.05$, respectively. Methods for automatically determining a suitable hyperparameter $l$ can be found in \cite[Chapter 5]{williams2006gaussian}.

\begin{table}[htb!]
  \centering
  \caption{Numerical comparison of the derivative estimations for $g(t) = \cos t$.  }\label{tab:analytical}
  \begin{tabular}{cccc}
  \toprule
  Numerical methods& $h = 0.01, \delta = 0.01$ & $h = 0.01, \delta = 0.1$ & $h = 0.1, \delta = 0.01$ \\ \toprule
  Finite difference & 1.71e-1  &1.56e-0 & 6.45e-2\\ \hline
  TV regularization & 1.90e-1  & 1.00e-1 & 4.01e-1  \\ \hline
  Method in \cite{nayak2020varitational} & 6.07e-2 & 8.39e-2 & 1.36e-1 \\ \hline
  vRKHS method & 1.86e-2 &7.49e-2 & 4.06e-2  \\ \bottomrule
  \end{tabular}%
\end{table}

In our first test, we consider the benchmark example of estimating the derivative of the function $g(x) = \cos t$ on $I = [-0.5, 0.5]$, as in \cite{knowles2014differentiation, nayak2020varitational}. The observation data is given by $y_i = g(t_i) + \eta_i$, where $t_i$ are the $n$ uniformly distributed points on $I$, and $\eta_i$ is Gaussian noise with standard deviation $\delta$. We quantify the performance of the numerical methods by considering the relative $L^2$ error $\|\dot{\phi} - \frac{\rd g}{\rd t}\|_{L^2(I)}/\|\frac{\rd g}{\rd t}\|_{L^2(I)}$, where $\dot{\phi}$ denotes the estimated derivative. Similar to \cite{knowles2014differentiation, nayak2020varitational}, we consider data points equidistributed on $I$ with mesh sizes $h = 0.1$ ($n = 11$) and $h = 0.01$ ($n = 101$). The hyperparameter $l$ in the Gaussian kernel is selected as $0.01$ and $0.1$ for $h = 0.01$, and $0.01$ for $h = 0.1$. \Cref{tab:analytical} displays the numerical errors of the three methods and compares the numerical results in \cite{nayak2020varitational}. Notably, our proposed method consistently achieves the smallest errors across all cases.

\begin{table}[htb!]
  \centering
  \caption{Numerical comparison of the derivative estimations from the noisy data for Lorenz63 system.}\label{tab:comp}
  \begin{tabular}{ccccc}
  \toprule
  Numerical methods& $\delta = 0.01$ & $\delta = 0.1$ & $\delta = 0.5$  & $\delta = 1$ \\ \toprule
  Finite difference & 2.50e-2  &2.45e-1 & 1.25e-0 & 2.48e-0\\ \hline
  TV regularization & 6.99e-2  & 2.12e-1 & 1.08e-0 & 1.56e-1 \\ \hline
  vRKHS method & 3.44e-3 &2.00e-2 & 7.31e-2 & 1.43e-1 \\ \bottomrule
  \end{tabular}
\end{table}

In our second test, the observation data is given by $\by = \bx + \bbeta$, where $\bbeta$ is a normal random variable in $\mathbb{R}^3$ with mean $0$ and standard deviation $\delta$, and $\bx$ is the numerical solution of the Lorenz63 model \cref{lorenz63_2} over $[0, 30]$. The observation data consists of uniformly spaced points with $h = 0.005$. We evaluate the performance by computing the discrete relative $L^2$ error, i.e., the discrete counterpart of $\|\dot{\bphi} - \frac{\rd\bx}{\rd t}\|_{L^2[0, 30]}/\|\frac{\rd\bx}{\rd t}\|_{L^2[0, 30]}$, where $\dot{\bphi}$ denotes the estimated derivative. We consider four different choices of $\delta = 0.01, 0.1, 0.5, 1$. The numerical errors are reported in \Cref{tab:comp}.

\begin{figure}[htbp]
	\centering
	\subfloat
	{\label{flcurve}\includegraphics[width=0.48\textwidth]{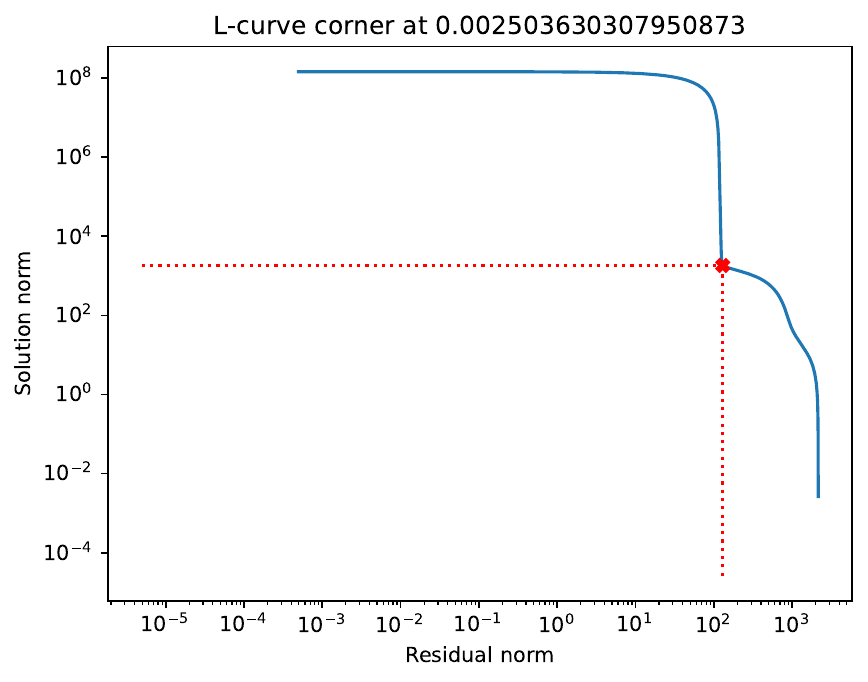}}
	\vspace{-3mm}
	\caption{Plot of L-curve for vRKHS based derivative estimation for the Lorenz63 system.}
	\label{fig:lcurve}
\end{figure}

From this table, we can see that our proposed method achieves a smaller relative error compared to the finite difference method and the TV regularization method when the noise level is small, i.e., $\delta = 0.01, 0.1$. In the large noise regime, our proposed method still turns out to be the best. Additionally, from \Cref{fig:lcurve}, we observe the L-shaped curve, where the best regularization parameter in our method can be automatically determined by finding the corner of the L-curve. In contrast, for the TV method, we have to try various $\lambda$ and solve \cref{TV_deriv} multiple times to select a good parameter.

\begin{figure}[htbp]
	\centering
	\subfloat
	{\label{fig:fd}\includegraphics[width=0.33\textwidth]{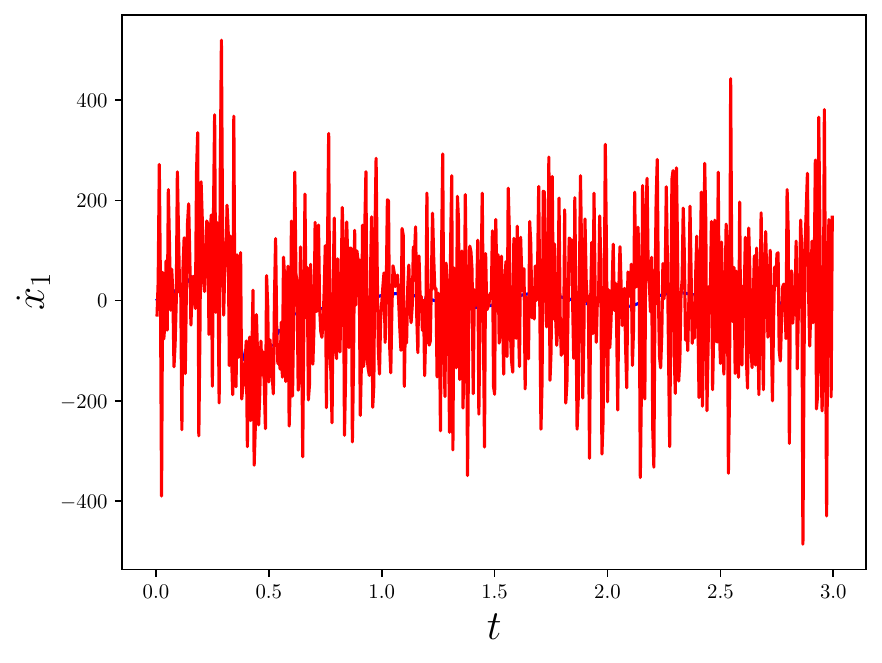}}\hspace{-1.0mm}
	\subfloat
	{\label{fig:tv}\includegraphics[width=0.33\textwidth]{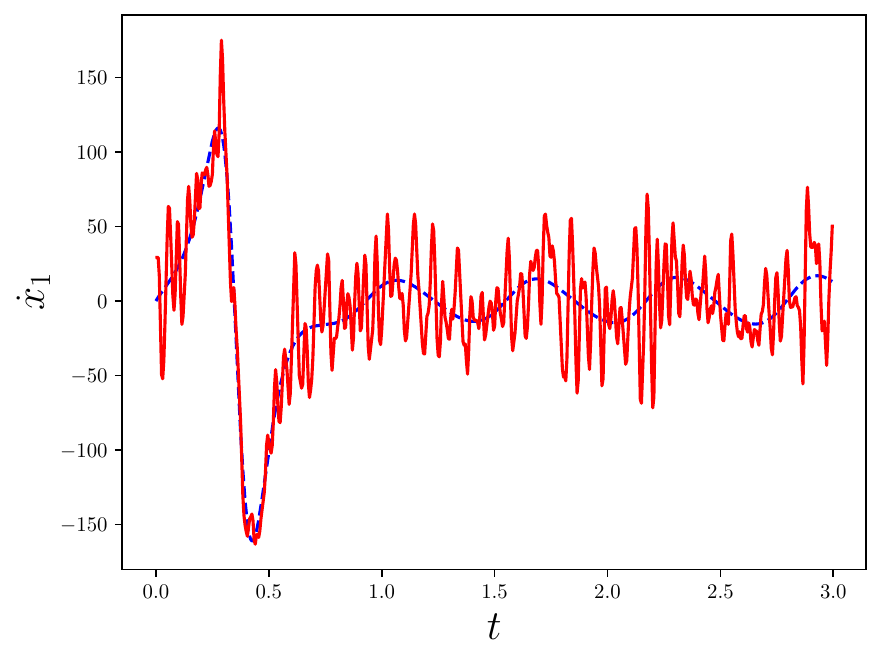}}\hspace{-1.0mm}
	\subfloat 
	{\label{fig:rkhs}\includegraphics[width=0.33\textwidth]{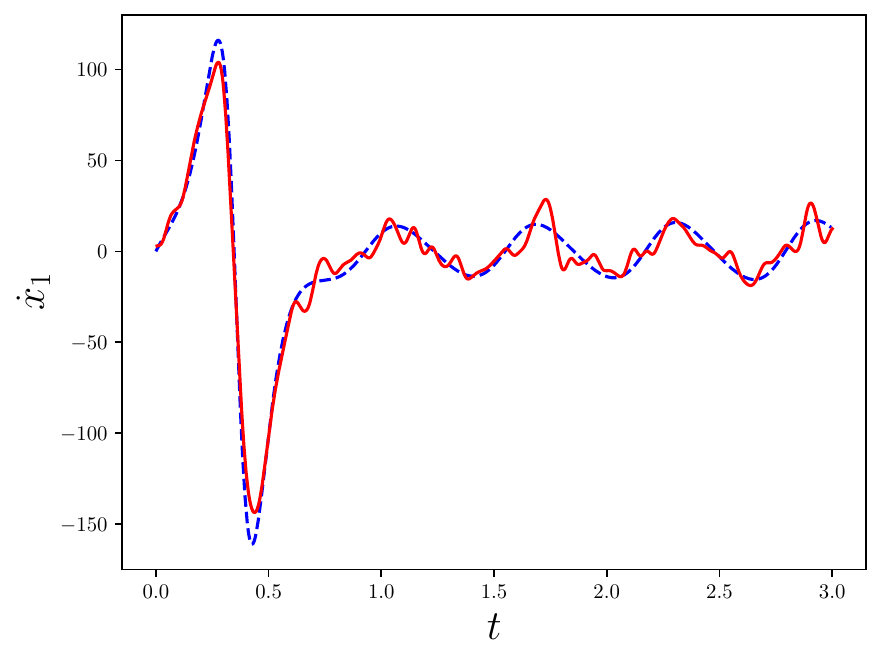}}
	\vspace{-3mm}
	\caption{Numerical derivative of $x_1$ from noise data for Lorzen63 system with $\delta = 1$. The exact derivative is shown in blue ($--$) and the numerical derivative is shown in red (--). (a) The numerical derivative obtained by finite difference method. (b) The numerical derivative obtained by TV regularization method. (c) The numerical derivative obtained by vRKHS method.}
	\label{fig:comp}
\end{figure}

To present a more intuitive comparison, we plot the numerical derivative of $\dot{x}_1$ over $(0, 3)$ in \Cref{fig:comp} when $\delta = 1$. From the figure, it is apparent that finite difference methods give inaccurate oscillatory numerical derivatives. A closer inspection of \Cref{fig:comp} shows that our proposed method provides highly accurate numerical differentiation, even in the presence of large noise.

\begin{figure}[htbp]
	\centering
	\subfloat
	{\label{fig:denoisedata_lorentz}\includegraphics[width=0.45\textwidth]{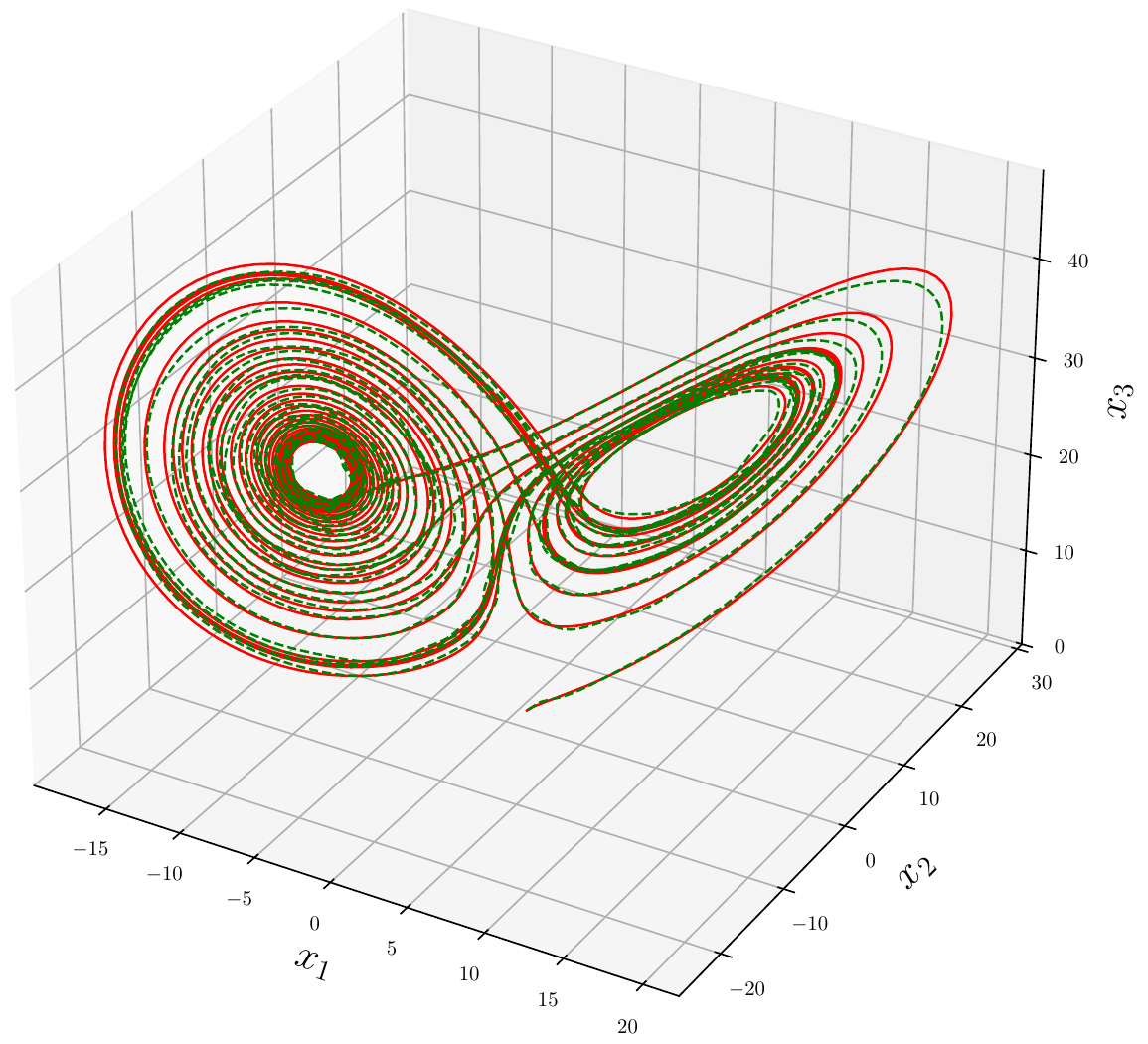}}\hspace{2.0mm}
	\subfloat
	{\label{fig:denoisedata_lotka}\includegraphics[width=0.40\textwidth]{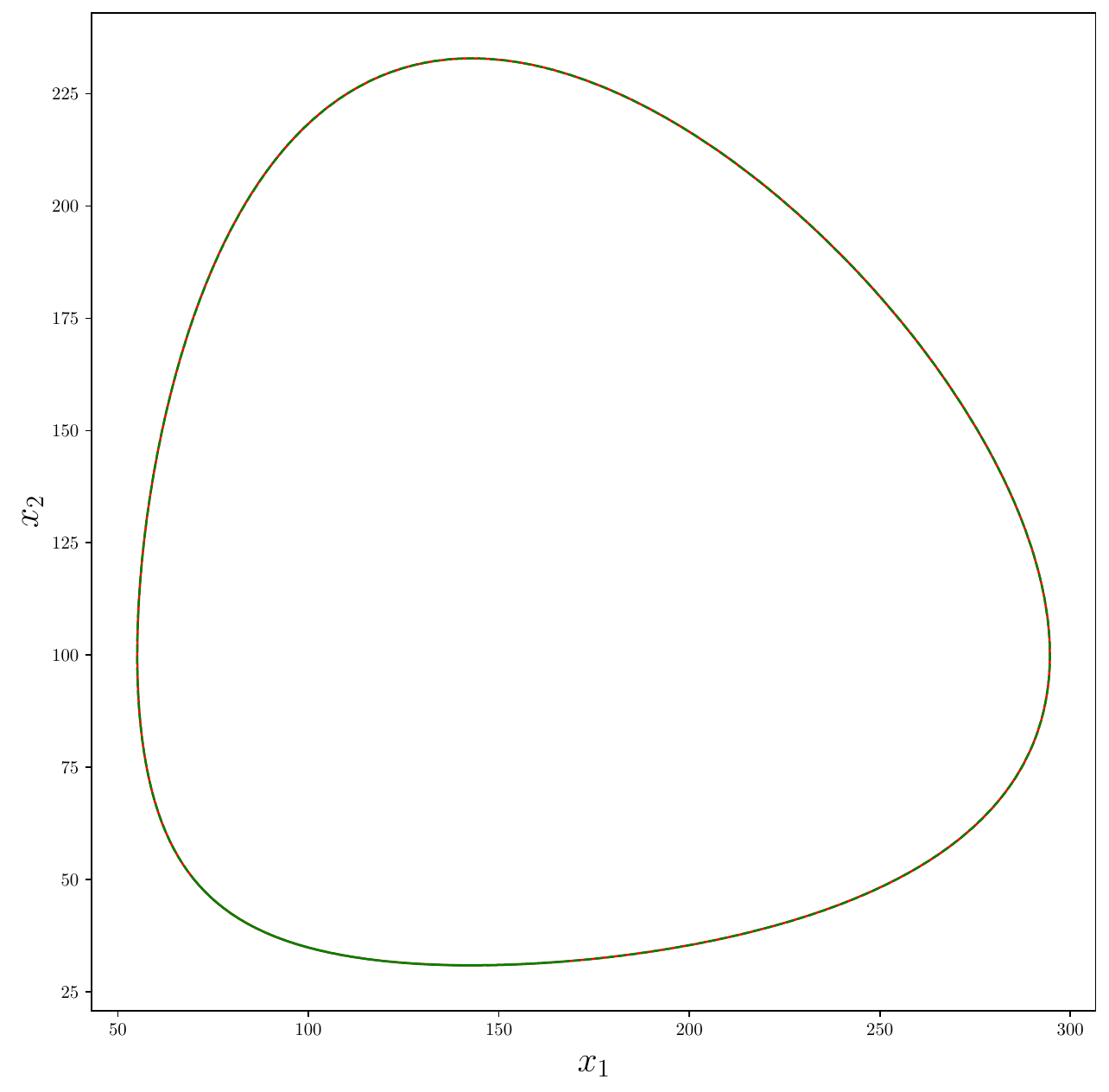}}\
	\vspace{-2mm}
	\caption{Observation data $y_1$ with noise $\delta = 1$ and the denoised data using vRKHS.  The exact data is shown in blue (--),  the observation data is shown in red (-), and the denoised data using vRKHS is shown in green ($--$). (a) Exact data and denoise data for Lorenz63 on equidistributed time points. (b) Exact data and denoised data for Lotka--Volterra system on nonuniform distributed time points.}
	\label{fig:denoise}
\end{figure}
    
\begin{figure}[htbp]
	\centering
	\subfloat
	{\label{fig:deriv}\includegraphics[width=1.0\textwidth]{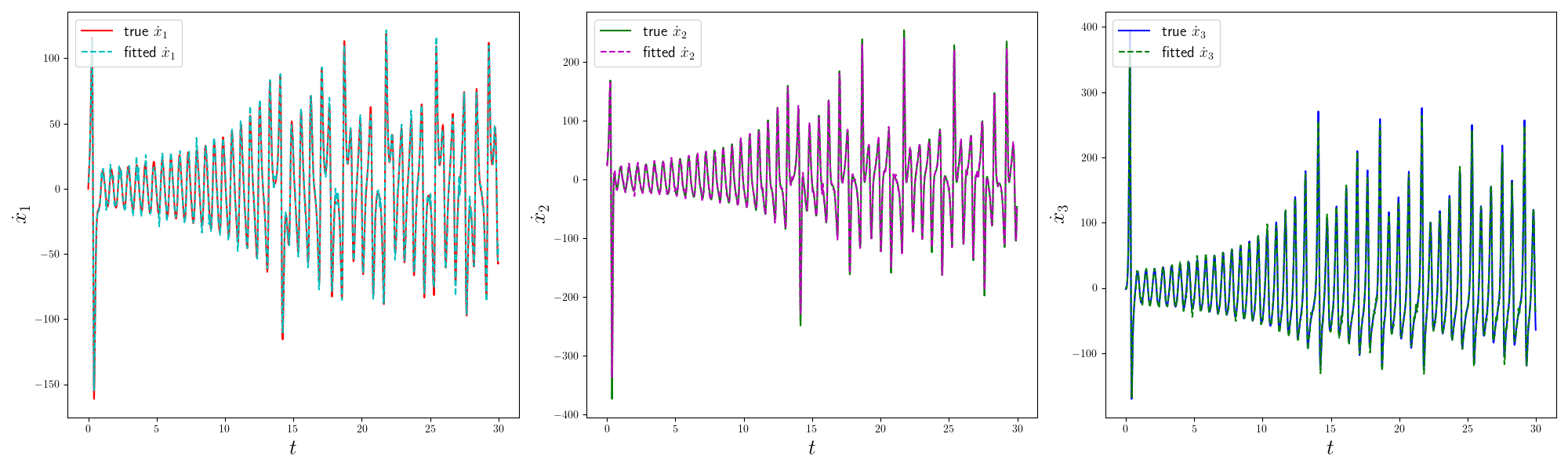}}\vspace{-2.0mm}
	\subfloat
	{\label{fig:traj}\includegraphics[width=1.0\textwidth]{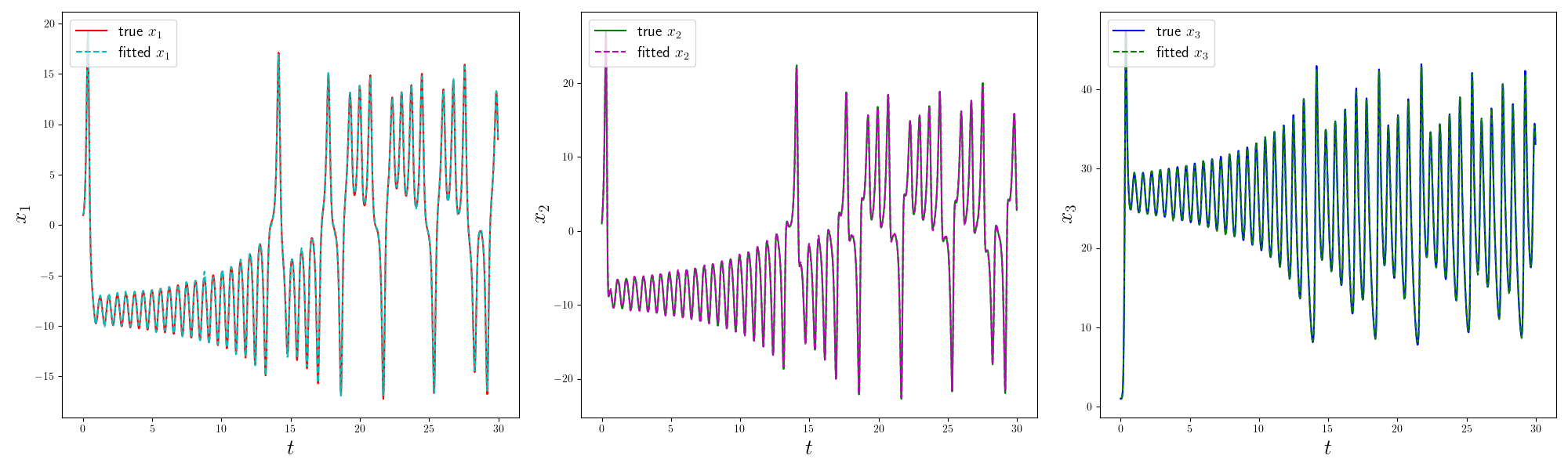}}
	\vspace{-3mm}
	\caption{True and fitted derivatives and trajectories of Lorenz63 system by vRKHS. The noise data are taken from 4000 equistributed points in $[0,30]$ with noise level $\delta = 0.5$.}
	\label{fig:fit}
\end{figure}

One of the most important distinguishing features of the proposed method is that we can fit the trajectory from the noisy data with only a little additional computational cost. We plot the noisy observation and the denoised data in \Cref{fig:denoise} when $\delta = 1$. Due to the large standard deviation $\delta$, the observation data is highly oscillatory, which leads to very large error of the numerical derivative obtained by the finite difference method. The denoised data obtained by the proposed method provide a very good fit to the exact data, as can be observed in \Cref{fig:denoisedata_lorentz}. A clearer presentation is provided in \Cref{fig:fit}, where we compare the true and fitted derivatives, as well as the trajectories, using 4000 noisy data points taken from the time interval $[0,30]$. From the figure, it is evident that, despite both $x_i$ and $\dot{x}_i$ being highly oscillatory functions, our vRKHS method is able to fit them well simultaneously from the noisy data.

\begin{table}[h]
  \centering
  \caption{Comparison of numerical errors for different methods on equistributed obervation data}
  \label{tab:other_odes}
  \begin{tabular}{ccccc}\toprule
    \multirow{2}{*}{ Numerical methods} & Pendulum & Lotka--Volterra & SIR & Lorenz96 \\
     \cline{2-5}
    & $\delta = 0.01$  & $\delta = 1$  &$\delta = 5 $ & $\delta = 0.1$ \\
      \hline
     Finite difference method & 7.71e-1 & 3.18e-0 & 1.40e+1 & 2.35e-0 \\ \hline
     TV regularization & 7.19e-1 & 2.15e-0 & 9.06e-0  & 6.71e-1 \\ \hline
     vRKHS method & 1.84e-2 & 1.83e-2 & 3.20e-2 & 2.46e-2 \\ \bottomrule
  \end{tabular}
\end{table} 

We also perform a comparative study for four other ODE models. Similar to the second numerical example, the observation data is obtained by adding white noise with standard deviation $\delta$ to the numerical solution at equidistributed data points with mesh size $h = 0.0125$ (Lorenz96), $h = 0.005$ (Lotka--Volterra), and $h = 0.01$ (Pendulum and SIR). The numerical results are listed in \Cref{tab:other_odes}. What stands out in this table is that the proposed method achieves the smallest numerical errors for all four ODE models.

\begin{table}[h]
  \centering
  \caption{Comparison of numerical errors for different methods on randomly distributed observation data}
  \label{tab:nonuniform_distributed}
  \resizebox{0.99\textwidth}{!}{
  \begin{tabular}{cccccc}
    \toprule
    \multirow{2}{*}{ Numerical methods} & Pendulum & Lotka--Volterra & Lorenz63  & SIR & Lorenz96 \\
     \cline{2-6}
    & $\delta = 0.01$  & $\delta = 1$ & $\delta =1$  &$\delta = 5 $ & $\delta = 0.1$ \\
      \hline
     Finite difference method & 9.03e-1 & 2.93e-0 & 1.75e-0 & 1.06e+1 & 8.97e-1 \\ \hline
     vRKHS method & 1.65e-2 & 1.91e-2 & 9.56e-2 & 3.20e-2 & 3.32e-2 \\
     \bottomrule
  \end{tabular}}
\end{table}

\begin{figure}[htbp]
	\centering
	\subfloat
	{\label{fig:lvfit}\includegraphics[width=1.0\textwidth]{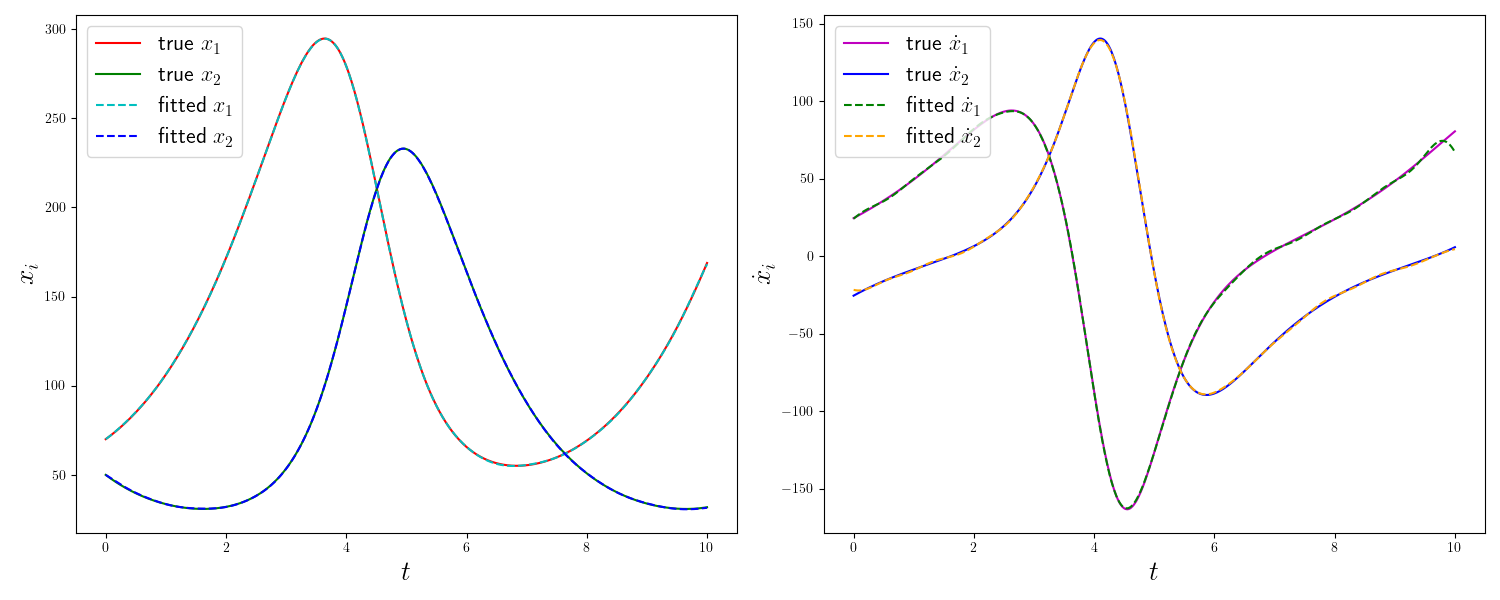}}
	\vspace{-5mm}
	\caption{True and fitted derivatives and trajectories of Lotka--Volterra system by vRKHS. The noise data are taken from 2000 random points in $[0,10]$ with noise level $\delta = 1.0$.}
	\label{fig:fit2}
\end{figure}

Another distinguished feature of the proposed method is that the observation data does not need to be equidistributed, which is frequently encountered in real applications. In contrast, the TV regularization method requires equidistributed data. In the rest of this subsection, we test the performance of the proposed methods on randomly distributed data and compare it with the finite difference method. In the test, we use the same number of observation data, but it is randomly distributed. The numerical results of the errors are reported in \Cref{tab:nonuniform_distributed}. Similar to the equidistributed observed data in \Cref{tab:other_odes}, the performance of the proposed method is consistent, with its error being much smaller than that of the finite difference method. In \Cref{fig:denoisedata_lotka}, we plot the denoised observation data for the Lotka-Volterra system. A clearer presentation is provided in \Cref{fig:fit2}, where we compare the true and fitted derivatives and trajectories. From the figure, we can clearly see that the denoised data is almost consistent with the exact data.

\subsection{Numerical results for learning dynamics}\label{sec62}
With the fitted derivative and trajectory, we can use either SINDy or \Cref{alg2:dynamics} to recover the dynamics $\bbf$. In the experiments that follow, we reconstructed the dynamics for the five examples from \cref{pendulum} to \cref{lorenz96}. The observation data used is identical to that shown in \Cref{fig:noisy_data1}. For the vRKHS method described in \Cref{alg2:dynamics}, we use the Gaussian kernel with the hyperparameter $l$ set to $1000$, $1000$, $1000$, $100$ and $100$ for the five examples, respectively.

\begin{table}[tbh] 
\caption{Accuracy of the reconstructed dynamics by SINDy and vRKHS, where $\bu$ represents the vector of parameters of the ODE. The $L^2$ norm of $\bbf$ for the five examples are taken in the regions $[0,0.4]\times[-0.4,0.4]$, $[50, 300]^2$, $[0,900]\times[10,600]\times[0,600]$, $[-20,20]\times[-20,20]\times[0,40]$, and $[-5,12]\times[-10,10]\times[-8,10]\times[-5,12]\times[-6,10]$, respectively.} 
\centering
\begin{tabular}{cccccc}
\toprule
Error & Pendulum & Lotka-Volterra & SIR & Lorenz63 & Lorenz96  \\
\toprule
$\frac{\|\bu-\bu_{\text{SINDy}\|_2}}{\|\bu\|_2}$ & \ \ -- & 1.34e-3 & \ \ -- & 8.12e-3 & 1.88e-3 \\
\hline 
$\frac{\|\bbf-\bbf_{\text{SINDy}}\|_{L^2}}{\|\bbf\|_{L^2}}$ & \ \ -- & 6.49e-4 & \ \ -- & 1.57e-2 & 6.05e-3 \\   
\hline 
$\frac{\|\bbf-\bbf_{\text{vRKHS}}\|_{L^2}}{\|\bbf\|_{L^2}}$ & 2.09e-2 & 1.22e-2 & 3.59e-1  & 3.96e-1 & 3.93e-2  \\         
\bottomrule
\end{tabular}\label{tab:dynamics}
\end{table}

The relative errors of the recovered parameters and the recovered dynamics using the SINDy and vRKHS methods are presented in \Cref{tab:dynamics}. We use $\bu$ to denote the vector of parameters for the ODEs, specifically: $\bu=(\alpha,\beta,\gamma,\delta)^{\top}$ for the Lotka-Volterra system, $\bu=(\sigma,\rho,\beta)^{\top}$ for the Lorenz63 system, and $\bu=F$ for the Lorenz96 systems. We note that although a set of basis functions can be constructed for SINDy to learn the dynamics of the nonlinear pendulum and SIR model, the true $\bbf$ is not a linear combination of basic elementary functions. Moreover, in many real-world applications, the parameterized expression of $\bbf$ is not known in advance. In such cases, kernel methods offer a more suitable approach for learning the dynamics. In the experiments, we only use the vRKHS method to recover $\bbf$ for the nonlinear pendulum system and SIR model. From \Cref{tab:dynamics}, we observe that the parameters recovered by both SINDy and vRKHS methods have good accuracy. The corresponding dynamics recovered by SINDy are more accurate than those obtained using the vRKHS method. This can be attributed to the fact that SINDy leverages more information about the dynamics, as the parameterized form of the ODE system is already known. In contrast, the vRKHS method only assumes that $\bbf$ belongs to a vRKHS, yet it is still able to recover $\bbf$ with relatively high accuracy.

\begin{figure}[htbp]
	\centering
	\subfloat[]
	{\label{fig:8a}\includegraphics[width=0.5\textwidth]{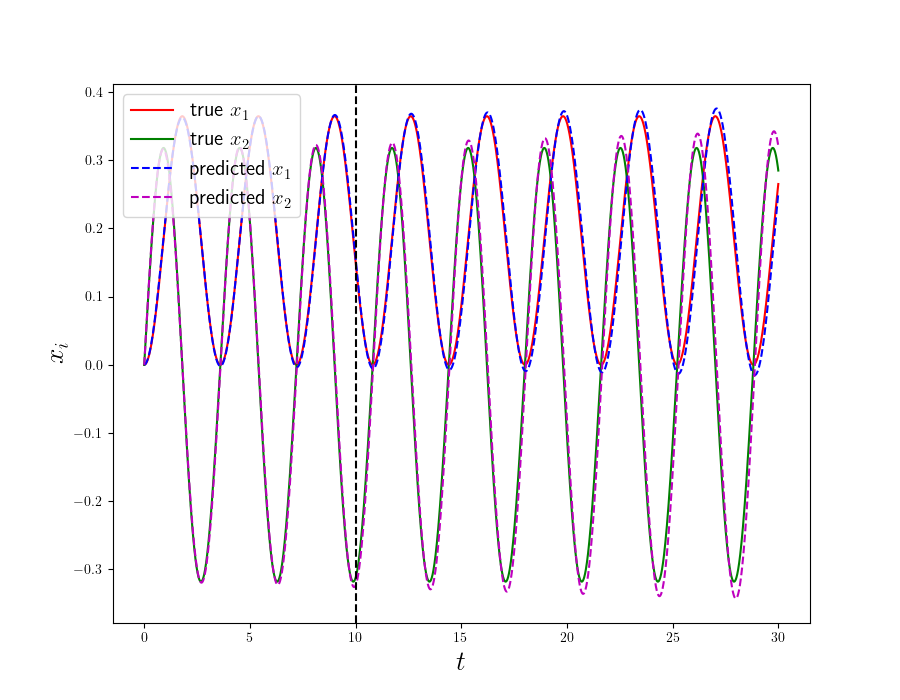}}\hspace{-2.0mm}
  \subfloat[]
  {\label{fig:8c}\includegraphics[width=0.45\textwidth]{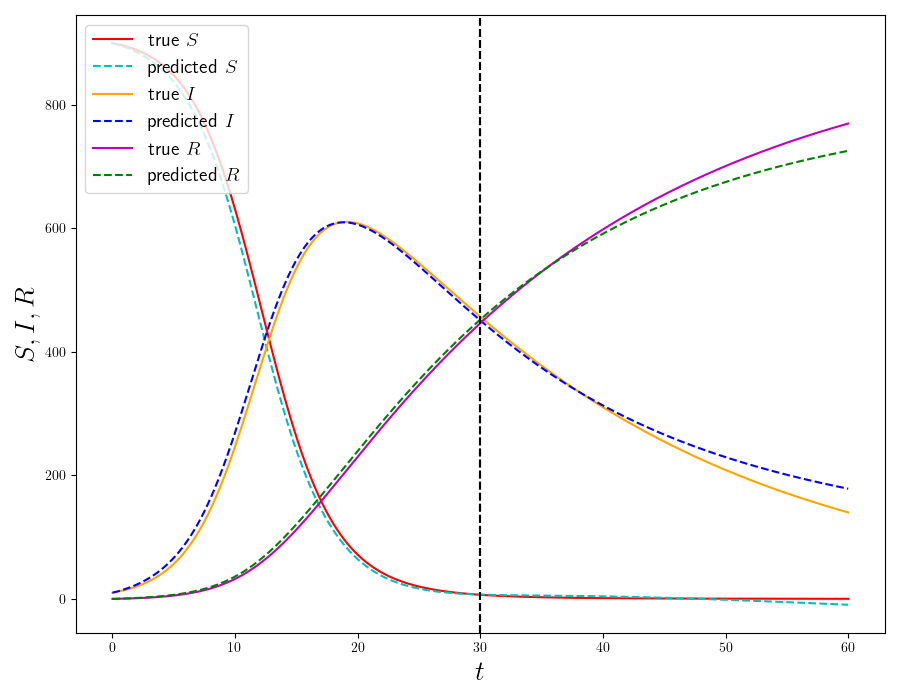}}
  \vspace{-2.0mm}
  \subfloat[]
	{\label{fig:9a}\includegraphics[width=1.0\textwidth]{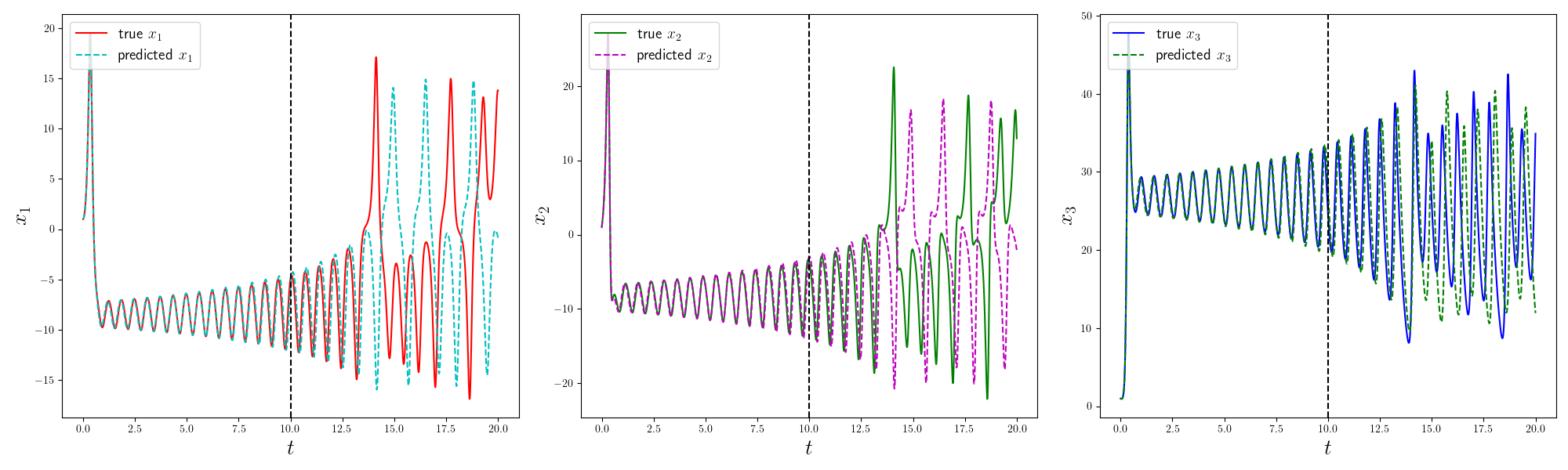}}
  \vspace{-2.0mm}
	\caption{Comparison of the true and predicted trajectories by vRKHS. (a) Forced vibration of nonlinear pendulum. (b) SIR model. (c) Lorenz63 system.}
	\label{fig:predict}
\end{figure}
After extracting the dynamics from the data, we can make predictions by solving the ODE using the recovered $\bbf$. The comparison between the true and predicted trajectories using the vRKHS method is shown in \Cref{fig:predict}, with a vertical dashed black line indicating the end time of the observation period. We present results for the nonlinear pendulum, SIR model, and Lorenz63 system; results for the other two examples are similar and are omitted for brevity. From the figures, we observe that for the nonlinear pendulum and SIR model, the recovered dynamics enable accurate predictions extending two to three times beyond the observation period. However, for the Lorenz63 system, the prediction become less accurate starting around $t \approx 14$; we remark that the result is similar for the recovered dynamics by SINDy. This discrepancy is likely due to the chaotic nature of the Lorenz63 system, as even a small perturbation in $\bbf$ can cause significant deviations in the trajectory from $t \approx 14$ onward.

\section{Conclusion}\label{sec7}
To mitigate the impact of noise on learning dynamics from data, we have proposed a method that simultaneously fits both the derivative and the trajectory from noisy time-series data. This approach treats the derivative estimation as an inverse problem involving integral operators in the forward model and estimates the derivative function by solving a vRKHS regularization problem. We have established an integral-form representer theorem for vRKHS, based on which we only need to compute the regularized solution by solving a finite-dimensional problem and can automatically select a good regularization parameter by the L-curve method. With the fitted derivative and trajectory, the dynamics can be recovered by solving a linear regularization problems by embedding the dynamics in a vRKHS. Several numerical experiments are performed to demonstrate the effectiveness and efficiency of our method.

When the data size $n$ is very large, the implementation of our method becomes challenging since $G_1$ and $\widetilde{G}_1$ are large-scale dense matrices and the eigen-decompositions are very expensive to compute. In this case, we could consider scalable RKHS methods to get sparse approximations of $G_1$ or $\widetilde{G}_1$, as referenced in \cite{quinonero2005unifying,liu2020gaussian}. Additional, iterative regularization methods that rely only on matrix-vector products are more efficient, as referenced in \cite{li2024preconditioned,li2024scalable}. These approaches will be the focus of our future research.

\section*{Acknowledgment}
This work was supported in part by the Andrew Sisson Fund and the Faculty Science Researcher Development Grant of the University of Melbourne.

\bibliographystyle{siamplain}
\bibliography{references}

\begin{thebibliography}{10}

\bibitem{aravkin2017generalized}
{\sc A.~Aravkin, J.~V. Burke, L.~Ljung, A.~Lozano, and G.~Pillonetto}, {\em
  Generalized {K}alman smoothing: Modeling and algorithms}, Automatica, 86
  (2017), pp.~63--86.

\bibitem{arnol2013mathematical}
{\sc V.~I. Arnold}, {\em Mathematical methods of classical mechanics}, vol.~60,
  Springer Science \& Business Media, 2013.

\bibitem{belytschko1996meshless}
{\sc T.~Belytschko, Y.~Krongauz, D.~Organ, M.~Fleming, and P.~Krysl}, {\em
  Meshless methods: an overview and recent developments}, Computer methods in
  applied mechanics and engineering, 139 (1996), pp.~3--47.

\bibitem{bongard2007automated}
{\sc J.~Bongard and H.~Lipson}, {\em Automated reverse engineering of nonlinear
  dynamical systems}, Proceedings of the National Academy of Sciences, 104
  (2007), pp.~9943--9948.

\bibitem{brock1991non}
{\sc W.~A. Brock and W.~D. Dechert}, {\em Non-linear dynamical systems:
  instability and chaos in economics}, Handbook of mathematical economics, 4
  (1991), pp.~2209--2235.

\bibitem{brunton2016discovering}
{\sc S.~L. Brunton, J.~L. Proctor, and J.~N. Kutz}, {\em Discovering governing
  equations from data by sparse identification of nonlinear dynamical systems},
  Proceedings of the national academy of sciences, 113 (2016), pp.~3932--3937.

\bibitem{butterworth1930theory}
{\sc S.~Butterworth et~al.}, {\em On the theory of filter amplifiers}, Wireless
  Engineer, 7 (1930), pp.~536--541.

\bibitem{carmeli2006vector}
{\sc C.~Carmeli, E.~De~Vito, and A.~Toigo}, {\em Vector valued reproducing
  kernel {H}ilbert spaces of integrable functions and {M}ercer theorem},
  Analysis and Applications, 4 (2006), pp.~377--408.

\bibitem{carmeli2010vector}
{\sc C.~Carmeli, E.~De~Vito, A.~Toigo, and V.~Umanit{\'a}}, {\em Vector valued
  reproducing kernel {H}ilbert spaces and universality}, Analysis and
  Applications, 8 (2010), pp.~19--61.

\bibitem{chartrand2011numerical}
{\sc R.~Chartrand}, {\em Numerical differentiation of noisy, nonsmooth data},
  International Scholarly Research Notices, 2011 (2011).

\bibitem{chen2018neural}
{\sc R.~T. Chen, Y.~Rubanova, J.~Bettencourt, and D.~K. Duvenaud}, {\em Neural
  ordinary differential equations}, Advances in neural information processing
  systems, 31 (2018).

\bibitem{crassidis2004optimal}
{\sc J.~L. Crassidis and J.~L. Junkins}, {\em Optimal estimation of dynamic
  systems}, Chapman and Hall/CRC, 2004.

\bibitem{cullum1971numerical}
{\sc J.~Cullum}, {\em Numerical differentiation and regularization}, SIAM
  Journal on numerical analysis, 8 (1971), pp.~254--265.

\bibitem{dixon2003nonlinear}
{\sc W.~E. Dixon, A.~Behal, D.~M. Dawson, and S.~P. Nagarkatti}, {\em Nonlinear
  control of engineering systems}, Boston, MA: Birk{\"a}user,  (2003).

\bibitem{favela2020dynamical}
{\sc L.~H. Favela}, {\em Dynamical systems theory in cognitive science and
  neuroscience}, Philosophy Compass, 15 (2020), p.~e12695.

\bibitem{freedman1980deterministic}
{\sc H.~I. Freedman}, {\em Deterministic mathematical models in population
  ecology}, (No Title),  (1980).

\bibitem{fuchs2014nonlinear}
{\sc A.~Fuchs}, {\em Nonlinear dynamics in complex systems}, Springer, 2014.

\bibitem{hansen1999curve}
{\sc P.~C. Hansen}, {\em The {L}-curve and its use in the numerical treatment
  of inverse problems}, IMM, Department of Mathematical Modelling, Technical
  University of Denmark,  (1999), pp.~119--142.

\bibitem{hokanson2023simultaneous}
{\sc J.~M. Hokanson, G.~Iaccarino, and A.~Doostan}, {\em Simultaneous
  identification and denoising of dynamical systems}, SIAM Journal on
  Scientific Computing, 45 (2023), pp.~A1413--A1437.

\bibitem{hu2022revealing}
{\sc P.~Hu, W.~Yang, Y.~Zhu, and L.~Hong}, {\em Revealing hidden dynamics from
  time-series data by {ODEN}et}, Journal of Computational Physics, 461 (2022),
  p.~111203.

\bibitem{janson2012non}
{\sc N.~B. Janson}, {\em Non-linear dynamics of biological systems},
  Contemporary Physics, 53 (2012), pp.~137--168.

\bibitem{kadri2016operator}
{\sc H.~Kadri, E.~Duflos, P.~Preux, S.~Canu, A.~Rakotomamonjy, and
  J.~Audiffren}, {\em Operator-valued kernels for learning from functional
  response data}, Journal of Machine Learning Research, 17 (2016), pp.~1--54.

\bibitem{kamalapurkar2018reinforcement}
{\sc R.~Kamalapurkar, P.~Walters, J.~Rosenfeld, and W.~Dixon}, {\em
  Reinforcement learning for optimal feedback control}, Springer, 2018.

\bibitem{kang2021ident}
{\sc S.~H. Kang, W.~Liao, and Y.~Liu}, {\em Ident: Identifying differential
  equations with numerical time evolution}, Journal of Scientific Computing, 87
  (2021), pp.~1--27.

\bibitem{katok1995introduction}
{\sc A.~Katok}, {\em Introduction to the modern theory of dynamical systems},
  Encyclopedia of Mathematics and its Applications, 54 (1995).

\bibitem{keesman2011system}
{\sc K.~J. Keesman}, {\em System identification: an introduction}, Springer
  Science \& Business Media, 2011.

\bibitem{kermack1927contribution}
{\sc W.~O. Kermack and A.~G. McKendrick}, {\em A contribution to the
  mathematical theory of epidemics}, Proceedings of the royal society of
  london. Series A, Containing papers of a mathematical and physical character,
  115 (1927), pp.~700--721.

\bibitem{knowles2014differentiation}
{\sc I.~Knowles and R.~J. Renka}, {\em Methods for numerical differentiation of
  noisy data}, in Proceedings of the {V}ariational and {T}opological {M}ethods:
  {T}heory, {A}pplications, {N}umerical {S}imulations, and {O}pen {P}roblems,
  vol.~21 of Electron. J. Differ. Equ. Conf., Texas State Univ., San Marcos,
  TX, 2014, pp.~235--246.

\bibitem{lahouel2024learning}
{\sc K.~Lahouel, M.~Wells, V.~Rielly, E.~Lew, D.~Lovitz, and B.~M. Jedynak},
  {\em Learning nonparametric ordinary differential equations from noisy data},
  Journal of Computational Physics, 507 (2024), p.~112971.

\bibitem{li2024preconditioned}
{\sc H.~Li}, {\em A preconditioned {K}rylov subspace method for linear inverse
  problems with general-form {T}ikhonov regularization}, SIAM Journal on
  Scientific Computing, 46 (2024), pp.~A2607--A2633.

\bibitem{li2024scalable}
{\sc H.~Li, J.~Feng, and F.~Lu}, {\em Scalable iterative data-adaptive {RKHS}
  regularization}, arXiv preprint arXiv:2401.00656,  (2024).

\bibitem{liu2020gaussian}
{\sc H.~Liu, Y.-S. Ong, X.~Shen, and J.~Cai}, {\em When {G}aussian process
  meets big data: A review of scalable {GP}s}, IEEE transactions on neural
  networks and learning systems, 31 (2020), pp.~4405--4423.

\bibitem{lorenz1963deterministic}
{\sc E.~N. Lorenz}, {\em Deterministic nonperiodic flow}, Journal of the
  atmospheric sciences, 20 (1963), pp.~130--141.

\bibitem{lorenz1996predictability}
{\sc E.~N. Lorenz}, {\em Predictability: A problem partly solved}, in Proc.
  Seminar on predictability, vol.~1, Reading, 1996.

\bibitem{lotka1925elements}
{\sc A.~J. Lotka}, {\em Elements of physical biology}, Williams \& Wilkins,
  1925.

\bibitem{lu2006numerical}
{\sc S.~Lu and S.~Pereverzev}, {\em Numerical differentiation from a viewpoint
  of regularization theory}, Mathematics of computation, 75 (2006),
  pp.~1853--1870.

\bibitem{messenger2021weak}
{\sc D.~A. Messenger and D.~M. Bortz}, {\em Weak {SIND}y: {G}alerkin-based
  data-driven model selection}, Multiscale Modeling \& Simulation, 19 (2021),
  pp.~1474--1497.

\bibitem{nayak2020varitational}
{\sc A.~Nayak}, {\em A new regularization approach for numerical
  differentiation}, Inverse Probl. Sci. Eng., 28 (2020), pp.~1747--1772,
  \url{https://doi.org/10.1080/17415977.2020.1763983},
  \url{https://doi.org/10.1080/17415977.2020.1763983}.

\bibitem{Nelles2001}
{\sc O.~Nelles}, {\em Nonlinear System Identification: From Classical
  Approaches to Neural Networks and Fuzzy Models}, Springer Berlin, Heidelberg,
  2001.

\bibitem{perko2013differential}
{\sc L.~Perko}, {\em Differential equations and dynamical systems}, vol.~7,
  Springer Science \& Business Media, 2013.

\bibitem{qin2019data}
{\sc T.~Qin, K.~Wu, and D.~Xiu}, {\em Data driven governing equations
  approximation using deep neural networks}, Journal of Computational Physics,
  395 (2019), pp.~620--635.

\bibitem{quinonero2005unifying}
{\sc J.~Quinonero-Candela and C.~E. Rasmussen}, {\em A unifying view of sparse
  approximate {G}aussian process regression}, Journal of machine learning
  research, 6 (2005), pp.~1939--1959.

\bibitem{rudy2019smoothing}
{\sc S.~H. Rudy, S.~L. Brunton, and J.~N. Kutz}, {\em Smoothing and parameter
  estimation by soft-adherence to governing equations}, Journal of
  Computational Physics, 398 (2019), p.~108860.

\bibitem{rudy2019deep}
{\sc S.~H. Rudy, J.~N. Kutz, and S.~L. Brunton}, {\em Deep learning of dynamics
  and signal-noise decomposition with time-stepping constraints}, Journal of
  Computational Physics, 396 (2019), pp.~483--506.

\bibitem{sastry2013nonlinear}
{\sc S.~Sastry}, {\em Nonlinear systems: analysis, stability, and control},
  vol.~10, Springer Science \& Business Media, 2013.

\bibitem{savitzky1964smoothing}
{\sc A.~Savitzky and M.~J. Golay}, {\em Smoothing and differentiation of data
  by simplified least squares procedures.}, Analytical chemistry, 36 (1964),
  pp.~1627--1639.

\bibitem{schmidt2009distilling}
{\sc M.~Schmidt and H.~Lipson}, {\em Distilling free-form natural laws from
  experimental data}, science, 324 (2009), pp.~81--85.

\bibitem{vogel2002computational}
{\sc C.~R. Vogel}, {\em Computational methods for inverse problems}, SIAM,
  2002.

\bibitem{vogel1996iterative}
{\sc C.~R. Vogel and M.~E. Oman}, {\em Iterative methods for total variation
  denoising}, SIAM Journal on Scientific Computing, 17 (1996), pp.~227--238.

\bibitem{volterra1928variations}
{\sc V.~Volterra}, {\em Variations and fluctuations of the number of
  individuals in animal species living together}, ICES Journal of Marine
  Science, 3 (1928), pp.~3--51.

\bibitem{wahba1990spline}
{\sc G.~Wahba}, {\em Spline models for observational data}, SIAM, 1990.

\bibitem{wang2011predicting}
{\sc W.-X. Wang, R.~Yang, Y.-C. Lai, V.~Kovanis, and C.~Grebogi}, {\em
  Predicting catastrophes in nonlinear dynamical systems by compressive
  sensing}, Physical review letters, 106 (2011), p.~154101.

\bibitem{wendland2004scattered}
{\sc H.~Wendland}, {\em Scattered data approximation}, vol.~17, Cambridge
  university press, 2004.

\bibitem{williams2006gaussian}
{\sc C.~K. Williams and C.~E. Rasmussen}, {\em Gaussian processes for machine
  learning}, vol.~2, MIT press Cambridge, MA, 2006.

\bibitem{wu2020structure}
{\sc K.~Wu, T.~Qin, and D.~Xiu}, {\em Structure-preserving method for
  reconstructing unknown {H}amiltonian systems from trajectory data}, SIAM
  Journal on Scientific Computing, 42 (2020), pp.~A3704--A3729.

\bibitem{wu2019numerical}
{\sc K.~Wu and D.~Xiu}, {\em Numerical aspects for approximating governing
  equations using data}, Journal of Computational Physics, 384 (2019),
  pp.~200--221.

\bibitem{wu2024characterizations}
{\sc Z.~Wu and Y.~Gao}, {\em Characterizations of linearly independent
  functions}, Journal of Mathematical Analysis, 15 (2024), pp.~55--68.

\bibitem{zhang2018robust}
{\sc S.~Zhang and G.~Lin}, {\em Robust data-driven discovery of governing
  physical laws with error bars}, Proceedings of the Royal Society A:
  Mathematical, Physical and Engineering Sciences, 474 (2018), p.~20180305.

\end{thebibliography}
\end{document}